\documentclass[10pt]{article}

\usepackage[margin=1in]{geometry}
\usepackage{amsmath,amssymb,amsthm,mathtools,mathrsfs}
\usepackage{booktabs}
\usepackage{enumitem}
\usepackage{microtype}
\usepackage{xcolor}
\usepackage[colorlinks=true,citecolor=blue!60!black,
  linkcolor=blue!60!black,urlcolor=blue!60!black]{hyperref}
\usepackage[nameinlink,noabbrev]{cleveref}

\allowdisplaybreaks
\numberwithin{equation}{section}
\setlist{nosep,leftmargin=*}

\newtheorem{theorem}{Theorem}[section]
\newtheorem{proposition}[theorem]{Proposition}
\newtheorem{lemma}[theorem]{Lemma}
\newtheorem{corollary}[theorem]{Corollary}

\theoremstyle{definition}

\newtheorem{example}[theorem]{Example}
\theoremstyle{remark}
\newtheorem{remark}[theorem]{Remark}

\DeclareMathOperator{\Tr}{Tr}
\DeclareMathOperator{\E}{\mathbb E}

\DeclareMathOperator{\Cov}{Cov}
\DeclareMathOperator{\Rea}{Re}
\DeclareMathOperator{\Ima}{Im}
\DeclareMathOperator{\Spec}{Spec}
\DeclareMathOperator{\dist}{dist}
\newcommand{\C}{\mathbb C}
\newcommand{\R}{\mathbb R}
\newcommand{\N}{\mathbb N}
\newcommand{\M}{\mathrm M}
\newcommand{\one}{\mathbf 1}
\newcommand{\tr}{\operatorname{tr}}

\newcommand{\cG}{\mathcal G}
\newcommand{\cL}{\mathcal L}

\newcommand{\norm}[1]{\left\lVert #1\right\rVert}
\newcommand{\abs}[1]{\left\lvert #1\right\rvert}
\newcommand{\ip}[2]{\left\langle #1,#2\right\rangle}
\newcommand{\transpose}{\mathsf T}

\hypersetup{
  pdftitle={Fourth-Moment Strong Universality for Finite-Type Transpose-Correlated Random Matrices},
  pdfauthor={Yanjin Xiang and Zhihua Zhang},
  pdfkeywords={strong convergence, universality, fourth moment, transpose correlation, variance profile, free probability}
}

\title{Fourth-Moment Strong Universality for Finite-Type\\
Transpose-Correlated Random Matrices}
\author{%
  Yanjin Xiang and Zhihua Zhang\\
  School of Mathematical Sciences, Peking University\\
  {\small\texttt{yjxiang@stu.pku.edu.cn; zhzhang@math.pku.edu.cn}}
}
\date{}

\begin{document}

\maketitle

\begin{abstract}
We prove a strong-universality theorem at the exact fourth-moment threshold
for finite families of non-Hermitian random matrices assembled from
independent unordered-pair vector atoms.  Within one atom, the matrix colors
and the two endpoint orientations may have arbitrary joint real covariance,
subject to reversal consistency across ordered type pairs and endpoint
exchangeability in same-type blocks; the law may also depend on finitely many
endpoint types.  The matrices may be
adjoined to an arbitrary deterministic tuple that converges jointly strongly
with the type projections.  For every fixed matrix amplification and fixed
noncommutative \(*\)-polynomial, the resulting tuple converges strongly to an
explicit covariance-matched free Gaussian family.  Cross-type blocks are
described by masked circular variables, whereas same-type blocks split into
independent endpoint-symmetric and endpoint-antisymmetric semicircular
sectors.  Only a finite radial fourth moment is assumed off the diagonal, and
a finite second moment suffices on the diagonal.  The mode of convergence
depends on the coupling: corners of one infinite array converge almost surely
on a common event, while fixed-law nonnested triangular arrays converge in
probability for each fixed test.  As an application, we obtain
exact-fourth-moment strong limits for finite-separable continuous left/right profiles,
including separately weighted literal-transpose terms.  For the nested
almost-sure formulation, the fourth-moment threshold is sharp already on the
Wigner subfamily; at this threshold arbitrary fresh rows admit no
coupling-invariant almost-sure upgrade of our in-probability conclusion.
\end{abstract}

\medskip
\noindent\textit{2020 Mathematics Subject Classification.}
60B20, 46L54.

\noindent\textit{Keywords.}
Strong convergence; universality; fourth moment; transpose correlation;
variance profile; operator-valued free probability.

\section{Introduction}
\label{sec:introduction}

Strong convergence of random matrices asks for substantially more than the
convergence of empirical eigenvalue distributions.  Given a random tuple
\(X_n\), it requires both normalized-trace convergence and operator-norm
convergence for every fixed noncommutative \(*\)-polynomial.  The norm part
detects spectral outliers and the edges of every self-adjoint polynomial
linearization; in a nonnormal problem it also controls the singular values of
polynomials that need not be normal.  This is different from a circular-law or
Brown-measure question, where the limiting eigenvalue measure is the principal
object and small singular values play a separate role.  The strong-convergence
program began with invariant ensembles and was subsequently extended to
several Wigner and deterministic-matrix settings
\cite{HaagerupThorbjornsen2005,CapitaineDonatiMartin2007,Male2012,
Anderson2013,BelinschiCapitaine2017}.

The independent object in the present paper is not a directed matrix entry.
For each unordered pair \(\{i,j\}\), we collect all colors in both directions,
\[
  \bigl(\xi_{ij}^{(1)},\ldots,\xi_{ij}^{(d)},
        \xi_{ji}^{(1)},\ldots,\xi_{ji}^{(d)}\bigr)\in\C^{2d}.
\]
These vectors are independent as \(\{i,j\}\) varies, but their coordinates
inside a single vector may be arbitrarily correlated subject to endpoint
reversal consistency; in a same-type block the pair law is exchangeable under
swapping the two orientations.  Thus the model contains colored elliptic-type
dependence, real and complex pseudocovariances, and correlation between a
matrix and its literal transpose.  The local law may
also depend on a finite type attached to each endpoint.  Treating the two
directed entries as independent would destroy precisely the covariance data
that we wish to retain.  Conversely, describing the covariance only as a
complex covariance matrix would discard pseudocovariance.  We therefore use
the complete covariance after identifying \(\C^{2d}\) with \(\R^{4d}\).

Our purpose is to determine the joint strong limit of these ensembles at the
finite-fourth-moment threshold, after adjoining deterministic structure.  If
\(P_{a,n}\) denotes the projection onto vertices of type \(a\), the required
deterministic hypothesis is joint strong convergence of
\[
  (P_{1,n},\ldots,P_{r,n},A_{1,n},\ldots,A_{m,n}).
\]
This joint assumption cannot be replaced by separate strong convergence of
the type projections and the companion tuple: mixed words in the two families
are part of the desired limit.  The comparison must preserve those mixed
words, survive self-adjoint linearization, and remain valid after tensoring
with an arbitrary but fixed matrix algebra.

\subsection*{Main contribution}

We prove that the full real covariance of the finitely many type-pair laws
determines the joint strong limit.  The limit is explicit.  For two different
types it is a finite linear combination of rectangular corners
\(p_a c p_b\) of circular elements.  Within one type, endpoint exchangeability
splits the covariance-matched Gaussian pair into an endpoint-symmetric sector
and an endpoint-antisymmetric sector.  These sectors are independent, and the
corresponding limit is a correctly normalized combination of two free
semicircular families.  The construction retains all correlations among
colors and orientations and is independent of the chosen real covariance
factorization.

The result has two deliberately different probabilistic formulations.  In the
\emph{nested} model, all matrices are upper-left corners of one infinite array
sampled once.  There is then one probability-one event on which convergence
holds for every fixed matrix level and every fixed polynomial.  In the
\emph{nonnested triangular} model, each size may be resampled from the same
finite family of type-pair laws.  Here convergence holds in probability,
separately for each fixed matrix level and polynomial.  No independence across
different rows is needed for this second statement.

This distinction is not cosmetic at the exact fourth moment.  A law with tail
of order
\(t^{-4}(\log t)^{-2}\) has finite fourth moment, while the probability of a
\(\sqrt n\)-scale entry among the order-\(n^2\) entries of a fresh row is
\((\log n)^{-2}\).  The probabilities vanish, which is consistent with
convergence in probability, but they need not be summable over independently
resampled rows.  In contrast, the nested proof uses the one-array structure
both for the original atoms and for the auxiliary type-block completions in
the tail argument.  For the nested almost-sure formulation, the fourth-moment
assumption is sharp already for the Wigner subfamily by the necessary part of
the Bai--Yin edge theorem \cite{BaiYin1988}.  The fresh-row example instead
shows why that almost-sure mode has no coupling-invariant extension at the
threshold; it does not contradict the fixed-law in-probability theorem.

At the deterministic level, our conclusion is joint strong \(*\)-convergence
at every fixed matrix amplification.  More precisely, for a prescribed
\(q<\infty\) and a prescribed \(M_q(\C)\)-valued \(*\)-polynomial \(Q\), we
obtain convergence of both
\[
  (\tr_q\otimes\tr_n)Q
  \quad\text{and}\quad
  \norm{Q}.
\]
The degree, coefficients, and amplification are fixed before
\(n\to\infty\).  We make no assertion uniform in \(q\), and in particular no
claim when \(q=q_n\to\infty\).

The finite-type theorem is designed as a reusable transfer principle rather
than as a single variance-profile calculation.  Its concrete application in
\cref{sec:profiles} starts from a homogeneous pair array and the deterministic
coordinate matrix
\(T_n=\operatorname{diag}(1/n,\ldots,n/n)\).  Finite sums of continuous
left/right multipliers applied to \(X_n\) and to the literal transpose
\(X_n^{\transpose}\) then converge strongly to a
\(C([0,1])\)-valued circular/elliptic family.  The transpose is installed as
an additional color at the unordered-pair level; it is not inferred from
strong convergence of \(X_n\) alone.  This application covers fixed
finite-separable masks.  It does not assert a theorem for arbitrary continuous
covariance kernels, continuum-indexed local laws, growing type counts, or
growing separable rank.

\subsection*{Relation to previous work}

An earlier preprint by the authors \cite{XiangZhang2026} treats the
homogeneous i.i.d.\ directed-entry model under the same fourth-moment
threshold, proving almost-sure joint strong \(*\)-convergence, including every
fixed matrix amplification, to a free circular family.  Its main
strong-convergence theorem is recovered from \cref{thm:nested-main} by taking
\(r=1\), omitting the deterministic companions, and choosing the
unordered-pair law to be the product of the prescribed entry marginals across
colors and orientations, with the diagonal law chosen accordingly.  The
present work extends that setting to finitely many endpoint types and
arbitrary dependence within each unordered-pair color/orientation vector,
subject to the model's reversal-consistency conditions; it also incorporates
jointly strongly convergent deterministic companions and adds the fixed-law
nonnested in-probability theorem and the finite-separable transpose-profile
application.  Although the proof follows the same
comparison--linearization--truncation architecture, the correlated finite-type
setting additionally requires identification of the covariance-matched
Gaussian endpoint and exact-$L^4$ tail control for correlated pair vectors,
including the endpoint-symmetric and endpoint-antisymmetric sectors within
same-type blocks.

Our point of departure is the covariance-matched universality theory of
Brailovskaya and van Handel \cite{BvH2024}.  Their comparison inequalities
apply to sums of independent self-adjoint random matrices, allowing arbitrary
dependence inside an individual summand and arbitrary deterministic means.
After self-adjoint linearization, one complete unordered-pair vector can
therefore be kept as one summand.  This is the quantitative engine of our
bounded comparison.  The abstract comparison theorem, however, does not by
itself identify the profiled Gaussian endpoint, remove truncation at exactly
four moments in every directed type block, or give the companion-stable
finite-type strong-limit statement proved here.  Related
concentration/free-probability connections for independent matrix sums appear in
\cite{BBvH2023}.

Exact-fourth-moment strong convergence in the presence of deterministic
matrices is already known for ordinary Wigner coordinates.  In particular,
Belinschi and Capitaine \cite{BelinschiCapitaine2017} establish polynomial
spectral and norm results for independent Wigner matrices with deterministic
companions, while Anderson \cite{Anderson2013} proves fourth-moment polynomial
norm convergence without the general companion structure used here.  These
works are important benchmarks and also supply sharp tail technology.  Our
extension concerns a different local independence unit: a type-dependent
vector atom that jointly couples all colors and both endpoint orientations.
Thus deterministic companions at the fourth-moment threshold are not, by
themselves, a novelty claim of this paper.

The covariance-matched Gaussian endpoint is treated using established
Gaussian strong-convergence results.  Male \cite{Male2012} provides the
unitarily invariant/deterministic-companion framework, and Fan--Sun--Wang
\cite{FanSunWang2021} give the GOE companion theorem that we use after a
doubling and corner construction.  The latter realizes real Ginibre sources,
and hence all cross-type and same-type Gaussian blocks, as fixed polynomials
in independent GOE matrices and one jointly convergent deterministic tuple.
The fixed matrix amplification in our statement is then obtained by a
separate complete-isometry argument; it is not attributed to the scalar GOE
theorem itself.

Classical elliptic-matrix results address a different level of convergence.
Naumov \cite{Naumov2012} and Nguyen--O'Rourke
\cite{NguyenORourke2015} prove elliptic laws for matrices whose opposite
entries are paired, while Adhikari--Bose \cite{AdhikariBose2019} obtain
asymptotic-freeness and Brown-measure results for elliptic and related
matrices, including deterministic matrices under their hypotheses.  These
results identify empirical eigenvalue or \(*\)-distribution limits; they do
not give joint operator-norm convergence for every polynomial with the
finite-type companions considered here.  Likewise, the deformed
variance-profile work of Alt--Kr\"uger \cite{AltKruger2024} identifies a
limiting Brown measure and its support through pseudospectra, whereas our
consequences concern spectra only for self-adjoint polynomial linearizations
and singular-value sets, not Brown measures of nonnormal polynomials.

There are also neighboring results with a different dependence structure or
endpoint.  Schultz \cite{Schultz2005} proves strong polynomial-norm limits for
independent GOE/GSE and related Gaussian ensembles.  Banna--C\'ebron
\cite{BannaCebron2023} obtain operator-valued Cauchy-transform convergence for
Wigner and Wishart matrices with free or exchangeable noncommutative entries,
and Reker \cite{Reker2022} proves stochastic operator-norm control for a
Hermitian ensemble with cumulant correlations decaying across matrix entries.
The first is Gaussian, the second has noncommutative entries and a
resolvent-distribution target, and the third permits longer-range dependence
but does not identify a companion-stable joint strong limit.  None supplies
the exact-fourth transfer theorem for classical vector atoms independent
across unordered pairs that is proved below.

The closest result on the operator-valued/profile axis is the Gaussian theory
of Jekel--Lee--Nelson--Pi \cite{JLNP2025}.  It covers broad operator-valued
covariance laws and continuously weighted Gaussian Wigner models, with
compatible deterministic approximants.  Its Gaussian covariance-map regime
is in important respects broader than our finite-separable application.  Our
contribution is instead the non-Gaussian transfer at the exact fourth moment
for finitely many unordered-pair laws.  For a single symmetric matrix with a
variance profile, Cheliotis--Louvaris \cite{CheliotisLouvaris2025} obtain
operator-norm limits under exact-fourth/Lindeberg assumptions in probability
and stronger hypotheses for their general almost-sure results.  Their theorem
does not contain joint polynomial strong convergence with arbitrary
companions or the colored endpoint-orientation covariance considered here,
but it precludes any broad claim that exact-fourth variance-profile norm
convergence is new.

Beyond the homogeneous independent-entry case of \cite{XiangZhang2026}, within
the class of classical matrix models that are independent across unordered
pairs, we are not aware of a theorem that simultaneously provides
exact-fourth-moment non-Gaussian universality, finitely typed joint
color/orientation covariance, arbitrary jointly strongly convergent
companions, and strong \(*\)-convergence at every fixed matrix level.  Existing
exact-fourth companion-stable theorems cover scalar Wigner coordinates, while
the broad companion-compatible operator-valued/profile results are Gaussian.
Our finite-type theorem bridges these regimes.  Each ingredient separately,
as well as elliptic weak limits and Gaussian transpose-correlated models, has
substantial prior art.

\subsection*{Proof strategy}

The proof has five stages.  First, a companion-augmented linearization reduces
fixed polynomial norm questions to spectra and moments of fixed self-adjoint
affine pencils.  At a fixed matrix level, all colors and both orientations
belonging to \(\{i,j\}\) are grouped into one centered self-adjoint summand.
For bounded atoms, the Brailovskaya--van Handel parameters satisfy
\[
  R=O(n^{-1/2}),\qquad \sigma=O(1),\qquad
  \sigma_*=O(n^{-1/2}),
\]
with constants allowed to depend on the fixed pencil and truncation level.
Their spectral and moment comparison estimates therefore transfer the
Gaussian pencil limit to the original bounded model.  The companion tuple
enters only through the deterministic mean and does not change these random
parameters.

Second, we identify the Gaussian comparator by factoring the complete real
covariance of every type-pair atom.  Cross-type factors are masked real
Ginibre matrices.  For a same-type block, endpoint exchangeability makes the
symmetric and antisymmetric Gaussian sectors independent.  A \(2n\)-dimensional
GOE doubling realizes the Ginibre sources as corners, and the GOE theorem of
Fan--Sun--Wang yields joint strong convergence with the lifted type projections
and companions.  A possible mismatch on the Gaussian diagonal is a diagonal
matrix of norm \(o(1)\) and hence does not alter the limit.

Third, centered radial truncation reduces the finite-fourth model to the
bounded one without changing endpoint-reversal consistency.  Cross-type
residual blocks are controlled by separate rectangular completions and the
Bai--Yin iid norm estimate, applied separately to real and imaginary parts
\cite{BaiYin1986}; same-type residuals are split into
symmetric and skew-symmetric sectors and controlled using both signs in
Anderson's fourth-moment theorem \cite{Anderson2013}.  These estimates use no
coordinate independence inside a pair atom.  The residual covariances vanish,
and principal covariance square roots couple the truncated free limits on a
common free Gaussian source.

Fourth, the comparison estimates are summable in the nested bounded model.
Together with nested auxiliary completions and a countable dense family of
rational pencils, this produces one common almost-sure event.  In a nonnested
row, the same estimates and tail controls give convergence in probability for
each fixed test.  Finally, polarization, faithfulness of the limiting trace,
and target-independent linearization upgrade pencil control to joint strong
\(*\)-convergence.  The fixed-matrix-level assertion follows from complete
isometry of the induced embedding into the norm ultraproduct.

The paper is organized as follows.  \Cref{sec:model-results} gives the model,
the explicit free Gaussian endpoint, and the two main theorems.
\Cref{sec:deterministic-tools} records the companion, amplification, and
linearization tools.  \Cref{sec:gaussian-comparator} proves the Gaussian
identification, and \cref{sec:bounded-transfer} performs the bounded
covariance-preserving comparison.  Exact-fourth truncation and the distinction
between the two couplings are treated in \cref{sec:exact-l4}.
\Cref{sec:profiles} proves the finite-separable profile application, and
\cref{sec:consequences} records spectral consequences and limitations.

\section{The finite-type model and main results}
\label{sec:model-results}

We now specify the independence structure, the deterministic data, and the
covariance-matched limit.  Throughout, \(d,r,m\geq1\) are fixed.  Neither the
number of colors nor the number of endpoint types varies with \(n\).

\subsection{Endpoint types and unordered-pair atoms}

Let \(\theta:\N\to[r]:=\{1,\ldots,r\}\) be deterministic, and define
\begin{equation}
  P_{a,n}:=\sum_{\substack{1\leq i\leq n\\ \theta(i)=a}}E_{ii}
  \in M_n(\C),
  \qquad a\in[r].
  \label{eq:type-projections}
\end{equation}
We assume that
\begin{equation}
  \tr_n(P_{a,n})\longrightarrow\pi_a>0,
  \qquad a\in[r],
  \label{eq:type-proportions}
\end{equation}
where \(\tr_n=n^{-1}\Tr\).  In particular, \(\sum_a\pi_a=1\).

Write
\[
  S:\C^d\oplus\C^d\longrightarrow\C^d\oplus\C^d,
  \qquad S(z_+,z_-)=(z_-,z_+)
\]
for endpoint reversal.  For every ordered type pair \((a,b)\), let
\(\mu^{ab}\) be a probability law on \(\C^{2d}\) such that
\begin{equation}
  \int v\,d\mu^{ab}(v)=0,
  \qquad
  \max_{a,b\in[r]}\int\norm{v}_2^4\,d\mu^{ab}(v)<\infty,
  \qquad
  \mu^{ba}=S_\#\mu^{ab}.
  \label{eq:pair-law-assumptions}
\end{equation}
Consequently, \(\mu^{aa}\) is invariant under endpoint exchange.  This last
condition is substantive: it rules out an orientation selected by the
numerical order \(i<j\) inside a same-type block.

In the nested model, for every \(i<j\) let
\begin{equation}
  V_{ij}:=
  \bigl(\xi_{ij}^{(1)},\ldots,\xi_{ij}^{(d)},
        \xi_{ji}^{(1)},\ldots,\xi_{ji}^{(d)}\bigr)
  \in\C^{2d}.
  \label{eq:pair-atom}
\end{equation}
The family \((V_{ij})_{i<j}\) is independent, is sampled once on one
probability space, and satisfies
\begin{equation}
  V_{ij}\sim\mu^{\theta(i)\theta(j)}.
  \label{eq:typed-pair-law}
\end{equation}
Dependence among the \(2d\) coordinates of a fixed \(V_{ij}\) is otherwise
arbitrary.  The reversal relation in \eqref{eq:pair-law-assumptions} makes the
description independent of the choice of canonical endpoint orientation.

For the diagonal, put
\[
  W_i=(\xi_{ii}^{(1)},\ldots,\xi_{ii}^{(d)})\in\C^d.
\]
We assume that the \(W_i\)'s are mutually independent, independent of all
off-diagonal atoms, centered, and have type-dependent laws \(\nu^{\theta(i)}\)
satisfying
\begin{equation}
  \max_{a\in[r]}\int\norm{w}_2^2\,d\nu^a(w)<\infty.
  \label{eq:diagonal-l2}
\end{equation}
No fourth moment is imposed on the diagonal.  Indeed, for each fixed color,
\begin{equation}
  \norm{n^{-1/2}\operatorname{diag}
    (\xi_{11}^{(\kappa)},\ldots,\xi_{nn}^{(\kappa)})}
  \longrightarrow0
  \label{eq:diagonal-negligible-preview}
\end{equation}
almost surely in the nested model and in probability in the fixed-law
nonnested model; see \cref{sec:exact-l4}.  Thus the diagonal has no component
in the limiting family.

The normalized random matrices are
\begin{equation}
  X_{\kappa,n}:=\frac1{\sqrt n}
  (\xi_{ij}^{(\kappa)})_{1\leq i,j\leq n},
  \qquad \kappa\in[d],
  \label{eq:random-matrices}
\end{equation}
and we abbreviate
\(X_n=(X_{1,n},\ldots,X_{d,n})\) and
\(P_n=(P_{1,n},\ldots,P_{r,n})\).

\begin{remark}[Why the atom is realified]
  The covariance data of \(V_{ij}\) mean the covariance of
  \((\Rea V_{ij},\Ima V_{ij})\in\R^{4d}\).  This records both the usual
  complex covariance and the pseudocovariance.  No circularity or independence
  of real and imaginary parts is assumed.
\end{remark}

\subsection{Deterministic companions and fixed amplified tests}

Let
\[
  A_n=(A_{1,n},\ldots,A_{m,n})\in M_n(\C)^m
\]
be deterministic.  We assume that \((P_n,A_n)\) converges jointly strongly in
\(*\)-distribution to
\[
  (p,a)=(p_1,\ldots,p_r,a_1,\ldots,a_m)
\]
in a faithful tracial \(C^*\)-probability space \((B,\tau_B)\).  Explicitly,
for every scalar noncommutative \(*\)-polynomial \(R\),
\begin{align}
  \tr_n R(P_n,A_n)&\longrightarrow\tau_B R(p,a),
  \label{eq:deterministic-trace-strong}\\
  \norm{R(P_n,A_n)}&\longrightarrow\norm{R(p,a)}.
  \label{eq:deterministic-norm-strong}
\end{align}
The limiting elements \(p_1,\ldots,p_r\) are pairwise orthogonal projections
with sum one and \(\tau_B(p_a)=\pi_a\).  The word ``jointly'' is essential in
\eqref{eq:deterministic-trace-strong}--\eqref{eq:deterministic-norm-strong};
separate strong limits of \(P_n\) and \(A_n\) do not determine mixed words.

For \(q\geq1\), write \(\tr_q=q^{-1}\Tr\).  An
\(M_q(\C)\)-valued \(*\)-polynomial is a finite sum
\(Q=\sum_j C_j\otimes R_j\), with \(C_j\in M_q(\C)\) and scalar
\(*\)-polynomials \(R_j\).  Its finite and limiting evaluations belong to
\(M_q\otimes M_n\) and \(M_q\otimes B\), respectively.  Every occurrence of
``fixed amplification'' below means that \(q\), \(Q\), its degree, and all
coefficients are chosen independently of \(n\).

\subsection{The covariance-matched free family}
\label{subsec:explicit-limit}

We next construct the limit \(x=(x_1,\ldots,x_d)\).  All covariance matrices
and Gaussian factorizations in this subsection are over real vector spaces.

\paragraph{Cross-type blocks.}
Fix \(a<b\).  Let
\(\Gamma^{ab}\in M_{4d}(\R)\) be the covariance matrix of the realification
of \(V\sim\mu^{ab}\).  Choose a real Gaussian factorization of its centered
Gaussian law,
\begin{equation}
  (Z^{ab}_{+,\kappa},Z^{ab}_{-,\kappa})_{\kappa\leq d}
  =\sum_{\ell=1}^{L_{ab}}
  (\alpha^{ab}_{\kappa\ell},
   \beta^{ab}_{\kappa\ell})_{\kappa\leq d}\,g^{ab}_\ell,
  \qquad L_{ab}\leq4d,
  \label{eq:cross-real-factorization}
\end{equation}
where the \(g^{ab}_\ell\)'s are independent real \(N(0,1)\) variables and
the displayed coefficients may be complex.  Equivalently, one may take the
principal square root of \(\Gamma^{ab}\) and delete zero columns.

On a reduced free-product extension of \((B,\tau_B)\), choose a standard
circular family \((c^{ab}_\ell)_{\ell\leq L_{ab}}\), free from \(B\).
We continue to denote the free-product trace, which extends the original
trace on \(B\), by \(\tau_B\).
Here ``standard'' means
\(\tau(c^{ab}_\ell(c^{ab}_\ell)^*)=1\).  Source families corresponding to
different unordered type pairs are chosen freely.  Set
\begin{equation}
  x^{ab}_\kappa
  :=\sum_{\ell=1}^{L_{ab}}
  \left(
    \alpha^{ab}_{\kappa\ell}
      p_a c^{ab}_\ell p_b
    +\beta^{ab}_{\kappa\ell}
      p_b(c^{ab}_\ell)^*p_a
  \right).
  \label{eq:cross-free-limit}
\end{equation}
The two terms use the same circular source; replacing one by an independent
copy would lose the orientation covariance.

\paragraph{Same-type blocks.}
Fix \(a\in[r]\), and let
\((Z^a_+,Z^a_-)\in\C^d\oplus\C^d\) be the centered real Gaussian vector
whose full real covariance equals that of \(\mu^{aa}\).  Define
\begin{equation}
  U^a=\frac{Z^a_++Z^a_-}{\sqrt2},
  \qquad
  V^a=\frac{Z^a_+-Z^a_-}{\sqrt2}.
  \label{eq:symmetric-antisymmetric-gaussians}
\end{equation}
Endpoint exchangeability sends \((U^a,V^a)\) to \((U^a,-V^a)\).  Hence
their complete real cross-covariance vanishes; because they are jointly real
Gaussian, \(U^a\) and \(V^a\) are independent.  Choose factorizations
\begin{equation}
  U^a_\kappa=\sum_{\ell=1}^{L_{a,+}}
       \alpha^a_{\kappa\ell}g^a_{+,\ell},
  \qquad
  V^a_\kappa=\sum_{\ell=1}^{L_{a,-}}
       \beta^a_{\kappa\ell}g^a_{-,\ell},
  \label{eq:same-real-factorization}
\end{equation}
where \(L_{a,+},L_{a,-}\leq2d\) and the two standard real Gaussian families
are independent.

Let \((s^a_{+,\ell})\) and \((s^a_{-,\ell})\) be standard semicircular
families, all mutually free, free from \(B\), and free from every cross-type
source.  Thus \(\tau((s^a_{\pm,\ell})^2)=1\).  With the skew-sector sign
convention fixed in \cref{sec:gaussian-comparator}, set
\begin{equation}
  x^{aa}_\kappa
  :=\frac1{\sqrt2}\,p_a
  \left(
    \sum_{\ell=1}^{L_{a,+}}
      \alpha^a_{\kappa\ell}s^a_{+,\ell}
    -i\sum_{\ell=1}^{L_{a,-}}
      \beta^a_{\kappa\ell}s^a_{-,\ell}
  \right)p_a.
  \label{eq:same-free-limit}
\end{equation}
The outer factor \(2^{-1/2}\) is compulsory: it inverts the change of
coordinates in \eqref{eq:symmetric-antisymmetric-gaussians}.  The independent
source families in \eqref{eq:same-free-limit} are a realization of the
Gaussian covariance decomposition, not an independence assumption on the
coordinates of the original atom.

Finally, define
\begin{equation}
  x_\kappa:=\sum_{a=1}^r x^{aa}_\kappa
       +\sum_{1\leq a<b\leq r}x^{ab}_\kappa,
  \qquad \kappa\in[d].
  \label{eq:total-free-limit}
\end{equation}
The projections in these formulas already encode the type proportions;
there is no additional factor \(\pi_a^{1/2}\) in
\eqref{eq:cross-free-limit} or \eqref{eq:same-free-limit}.

\begin{proposition}[Factorization invariance and Gaussian identification]
  \label{prop:limit-identification}
  The joint \(*\)-moments of \((p,a,x)\), as well as all scalar and
  matrix-amplified polynomial norms, are independent of the real covariance
  factorizations in \eqref{eq:cross-real-factorization} and
  \eqref{eq:same-real-factorization}.  Moreover, the covariance-matched
  finite-type Gaussian matrix tuple, adjoined to \((P_n,A_n)\), converges
  jointly strongly almost surely, on a single event, to \((p,a,x)\).  Any type-dependent Gaussian diagonal
  convention changes the GOE-polynomial realization by a diagonal matrix of
  norm \(o(1)\) almost surely and hence does not change the limit.
\end{proposition}

The proposition is proved in \cref{sec:gaussian-comparator}.  Briefly,
orthogonal changes of real Gaussian coordinates rotate free Gaussian source
families, while a doubled GOE realization and compression produce the
circular and semicircular variables above.  Notice that the sources are free
from \(B\); the compressed tuple \(x\) itself is generally not scalar-free
from \(B\) because of the type projections.

\subsection{Main strong-universality theorems}

We first state the result for one infinite array.  The notation
\(Q(P_n,A_n,X_n)\) includes adjoints of every displayed variable whenever
they occur in \(Q\).

\begin{theorem}[Nested infinite-array universality]
  \label{thm:nested-main}
  Assume \eqref{eq:type-proportions}--\eqref{eq:diagonal-l2} and the joint
  deterministic strong convergence
  \eqref{eq:deterministic-trace-strong}--\eqref{eq:deterministic-norm-strong}.
  Suppose that all sizes are the upper-left corners of the single infinite
  array described in \eqref{eq:pair-atom}--\eqref{eq:typed-pair-law}, with the
  diagonal variables sampled once as well.  Then there exists an event
  \(\Omega_0\) of probability one such that, on \(\Omega_0\), for every
  fixed \(q\geq1\) and every fixed \(M_q(\C)\)-valued noncommutative
  \(*\)-polynomial \(Q\),
  \begin{align}
    (\tr_q\otimes\tr_n)Q(P_n,A_n,X_n)
      &\longrightarrow
    (\tr_q\otimes\tau_B)Q(p,a,x),
    \label{eq:nested-trace-convergence}\\
    \norm{Q(P_n,A_n,X_n)}
      &\longrightarrow
    \norm{Q(p,a,x)}.
    \label{eq:nested-norm-convergence}
  \end{align}
  The limit \(x\) is the covariance-matched family
  \eqref{eq:cross-free-limit}--\eqref{eq:total-free-limit}.
\end{theorem}

The common event in \cref{thm:nested-main} is independent of \((q,Q)\).
This statement for all complex coefficients follows from a countable dense
family of rational-coefficient pencils, uniform boundedness on that event,
and polynomial continuity.  It does not provide estimates uniform over the
tests.

We next formulate the coupling-invariant fixed-law statement.  For every
\(n\), let
\[
  V_{ij}^{(n)}=
  (\xi_{ij}^{(1,n)},\ldots,\xi_{ij}^{(d,n)},
   \xi_{ji}^{(1,n)},\ldots,\xi_{ji}^{(d,n)}),
  \qquad 1\leq i<j\leq n.
\]
Within row \(n\), assume that these atoms are independent and that
\(V_{ij}^{(n)}\sim\mu^{\theta(i)\theta(j)}\).  The diagonal vectors within
that row are mutually independent, independent of the off-diagonal atoms,
and have laws \(\nu^{\theta(i)}\).  No condition is imposed on the joint
coupling of different rows.

\begin{theorem}[Fixed-law nonnested triangular arrays]
  \label{thm:triangular-main}
  Under the preceding rowwise assumptions and the same deterministic
  hypotheses as in \cref{thm:nested-main}, define
  \(X_{\kappa,n}=n^{-1/2}(\xi_{ij}^{(\kappa,n)})_{i,j\leq n}\).
  For every fixed \(q\geq1\) and every fixed
  \(M_q(\C)\)-valued noncommutative \(*\)-polynomial \(Q\),
  \begin{align}
    (\tr_q\otimes\tr_n)Q(P_n,A_n,X_n)
      &\longrightarrow
    (\tr_q\otimes\tau_B)Q(p,a,x)
      &&\text{in probability},
    \label{eq:triangular-trace-convergence}\\
    \norm{Q(P_n,A_n,X_n)}
      &\longrightarrow
    \norm{Q(p,a,x)}
      &&\text{in probability}.
    \label{eq:triangular-norm-convergence}
  \end{align}
  Here \(x\) is again the explicit covariance-matched family
  \eqref{eq:cross-free-limit}--\eqref{eq:total-free-limit}.
\end{theorem}

Theorems \ref{thm:nested-main} and \ref{thm:triangular-main} differ only in
their coupling and mode of convergence, not in their moment assumptions or
their limit.  In particular, resampling from the same fixed finite family of
laws does not require a uniform-integrability assumption beyond
\eqref{eq:pair-law-assumptions} and \eqref{eq:diagonal-l2}.

\begin{remark}[Special cases]
  If the deterministic base is only
  \(D=\operatorname{span}\{p_1,\ldots,p_r\}\cong\C^r\), the self-adjoint
  coordinates of \(x\) form the corresponding \(D\)-valued free Gaussian
  system; its covariance maps are read directly from the type-pair second
  moments and the corner traces \(\tau_B(p_a)=\pi_a\).  When \(r=1\), the
  projections disappear and the limit is the scalar free Gaussian family
  determined by the color/orientation covariance, free from the companion
  limit.  For \(r>1\), scalar freeness of \(x\) from \(B\) is not asserted.
\end{remark}

\subsection{Finite-separable profile application: preview}

The formal profile theorem is stated and proved in \cref{sec:profiles}; we
record its scope here to make clear what follows from the abstract result.
Take \(r=1\) and
\[
  T_n=\operatorname{diag}(1/n,2/n,\ldots,n/n),
\]
which converges strongly to the coordinate function \(t\) in
\(C([0,1])\) equipped with Lebesgue trace.  For finitely many colors,
outputs, and summands, let
\begin{equation}
\begin{split}
  Z_{\rho,n}
  ={}&\sum_{\kappa,\ell}
       a_{\rho\kappa\ell}(T_n)X_{\kappa,n}
       b_{\rho\kappa\ell}(T_n)\\
    &+\sum_{\kappa,\ell}
       \widetilde a_{\rho\kappa\ell}(T_n)
       X_{\kappa,n}^{\transpose}
       \widetilde b_{\rho\kappa\ell}(T_n),
  \label{eq:profile-preview}
\end{split}
\end{equation}
where every scalar mask is continuous on \([0,1]\), and every sum has rank
fixed independently of \(n\).  Apply the main theorem to the enlarged color
tuple
\((X_{1,n},\ldots,X_{d,n},X_{1,n}^{\transpose},\ldots,
X_{d,n}^{\transpose})\).  Its pair atom is a singular deterministic image in
\(\R^{8d}\) of the original \(\R^{4d}\) atom, so the transpose coordinate
uses the same free sources with the endpoint-reversal sign, not an independent
copy.  Polynomial approximation of the continuous masks then gives the same
convergence mode as the applicable main theorem.  Extra deterministic
companions are permitted only when they converge jointly strongly with
\(T_n\).  This argument applies to the fixed finite-separable form
\eqref{eq:profile-preview}; uniform density alone does not extend it to an
arbitrary covariance kernel without an additional operator-norm estimate.

\section{Deterministic companions, amplification, and linearization}
\label{sec:deterministic-tools}

This section isolates the deterministic operator-algebraic steps that turn
control of self-adjoint affine pencils into joint strong convergence.  These
steps do not depend on the distribution of the random entries.  In
particular, they are logically separate from the covariance comparison used
in Section~\ref{sec:bounded-transfer}.

\subsection{Self-adjoint coordinates and companion-augmented pencils}
\label{subsec:self-adjoint-pencils}

Let \((\mathcal A,\tau)\) be a tracial \(C^*\)-probability space with
faithful trace.  For \(q\geq 1\), write
\[
 \tau_q=\tr_q\otimes\tau,
 \qquad \tr_q=q^{-1}\Tr.
\]
A tuple \(Y_n=(Y_{1,n},\ldots,Y_{s,n})\) in tracial matrix
algebras \((\M_{N_n}(\C),\tr_{N_n})\) converges \emph{strongly in joint
\(*\)-distribution} to \(y=(y_1,\ldots,y_s)\in\mathcal A^s\) if, for every
scalar noncommutative \(*\)-polynomial \(P\),
\begin{equation}
 \tr_{N_n}P(Y_n)\longrightarrow\tau(P(y)),
 \qquad
 \norm{P(Y_n)}\longrightarrow\norm{P(y)}.
 \label{eq:def-strong-convergence}
\end{equation}
For a non-self-adjoint coordinate \(z\), we use the self-adjoint coordinates
\[
 \Rea z=\frac{z+z^*}{2},
 \qquad
 \Ima z=\frac{z-z^*}{2i}.
\]
Thus no generality is lost by formulating the intermediate arguments for
self-adjoint tuples.

Fix a matrix-coefficient size \(q\).  After this replacement, every affine
pencil needed below has the form
\begin{equation}
 \cL_n
 =C_0\otimes I_{N_n}
  +\sum_{j=1}^{s_0}C_j\otimes D_{j,n}
  +\sum_{k=1}^{s_1}B_k\otimes H_{k,n},
 \qquad C_j,B_k\in\M_q(\C)_{\rm sa}.
 \label{eq:companion-augmented-pencil}
\end{equation}
Here \(D_n\) consists of the self-adjoint coordinates of the type projections
and deterministic companions, while \(H_n\) consists of the random
self-adjoint coordinates.  After the norm-negligible random diagonal has been
removed, the last sum in the present model can be grouped by unordered pairs
as
\begin{equation}
 \cL_n=\cL_n^{(0)}+\sum_{1\leq i<j\leq n}Z_{ij,n},
 \qquad Z_{ij,n}=Z_{ij,n}^*,\quad \E Z_{ij,n}=0,
 \label{eq:pencil-pair-decomposition}
\end{equation}
where the matrices \((Z_{ij,n})_{i<j}\) are independent.  The deterministic
matrix \(\cL_n^{(0)}\) contains the first two terms of
\eqref{eq:companion-augmented-pencil} and any other deterministic shift.

This separation explains precisely how companions enter a covariance
comparison.  For example, the three random-summand parameters
\begin{align}
 R(\cL_n)&=\max_{i<j}\norm{Z_{ij,n}}_{L^\infty},
 \label{eq:pencil-R}\\
 \sigma(\cL_n)^2&=\norm{\sum_{i<j}\E Z_{ij,n}^2},
 \label{eq:pencil-sigma}\\
 \sigma_*(\cL_n)^2
 &=\sup_{\norm{u}=\norm{v}=1}
   \sum_{i<j}\E\abs{\ip{u}{Z_{ij,n}v}}^2
 \label{eq:pencil-sigma-star}
\end{align}
do not involve \(\cL_n^{(0)}\).  Consequently, adjoining deterministic
companions changes neither the covariance-matching requirement nor these
parameters.  It does, however, change the deterministic target of the pencil.
This is why the hypotheses of our main results require \emph{joint} strong
convergence of the type projections and companions; separate strong
convergence of two deterministic tuples would not identify the norms of
mixed polynomials.

For later use, if \(Q\in\M_q(\C\langle z,z^*\rangle)\) is not self-adjoint,
then
\begin{equation}
 \mathfrak h(Q)=
 \begin{pmatrix}0&Q\\Q^*&0\end{pmatrix}
 \quad\text{satisfies}\quad
 \mathfrak h(Q)=\mathfrak h(Q)^*,
 \qquad \norm{\mathfrak h(Q)}=\norm{Q}.
 \label{eq:hermitization-polynomial}
\end{equation}
Hence self-adjoint matrix-valued polynomial estimates suffice for arbitrary
polynomial norms, at the cost of replacing \(q\) by the still fixed size
\(2q\).

\subsection{Fixed matrix amplification}
\label{subsec:fixed-amplification}

The Gaussian strong-convergence theorem used in
Section~\ref{sec:gaussian-comparator} is a scalar-level statement.  The next
lemma, rather than that theorem, supplies matrix-valued coefficients.

\begin{lemma}[Fixed-amplification lemma]
\label{lem:fixed-amplification}
Suppose that \(Y_n\) converges strongly in joint \(*\)-distribution to
\(y\) as in \eqref{eq:def-strong-convergence}.  Then, for each separately
fixed \(q\geq1\) and every
\(Q\in\M_q(\C\langle z,z^*\rangle)\),
\begin{align}
 \norm{Q(Y_n)}&\longrightarrow\norm{Q(y)},
 \label{eq:amplified-norm-convergence}\\
 (\tr_q\otimes\tr_{N_n})Q(Y_n)
 &\longrightarrow(\tr_q\otimes\tau)Q(y).
 \label{eq:amplified-trace-convergence}
\end{align}
The assertion is pointwise in \(q\).  It gives no estimate uniform over
matrix levels and no conclusion when \(q=q_n\to\infty\).
\end{lemma}

The proof, given in Appendix~\ref{app:amplification-linearization}, embeds
the generated limit algebra into every \(C^*\)-ultraproduct of the matrix
algebras.  An injective \(*\)-homomorphism between \(C^*\)-algebras is
completely isometric, which yields
\eqref{eq:amplified-norm-convergence};
\eqref{eq:amplified-trace-convergence} is entrywise.

We will apply Fan--Sun--Wang's GOE theorem
\cite[Theorem~4.3]{FanSunWang2021} exactly at scalar level.  That theorem is
stated for self-adjoint \(*\)-polynomials.  If \(Q\) is arbitrary, its norm
conclusion follows by applying the theorem to the self-adjoint polynomial
\(Q^*Q\):
\[
 \norm{Q(Y_n)}^2=\norm{Q(Y_n)^*Q(Y_n)}.
\]
Trace convergence for arbitrary \(Q\) follows by applying the self-adjoint
statement to \(\Rea Q\) and \(\Ima Q\).  Lemma~\ref{lem:fixed-amplification}
is then applied separately for each fixed matrix level; no matrix-amplified
statement is attributed to Fan--Sun--Wang.

\subsection{From affine pencils to strong convergence}
\label{subsec:pencil-criterion}

We now state the deterministic passage used after the probabilistic
comparison has controlled all fixed affine pencils.  For a self-adjoint tuple
\(Y_n=(Y_{1,n},\ldots,Y_{s,n})\), fixed
\(A_0,\ldots,A_s\in\M_q(\C)_{\rm sa}\), and a self-adjoint target tuple
\(y\), set
\begin{align}
 L_n(A)&=A_0\otimes I_{N_n}+\sum_{j=1}^sA_j\otimes Y_{j,n},
 \label{eq:affine-pencil-finite}\\
 L(A)&=A_0\otimes\one+\sum_{j=1}^sA_j\otimes y_j.
 \label{eq:affine-pencil-limit}
\end{align}

\begin{proposition}[Companion-augmented affine-pencil criterion]
\label{prop:affine-pencil-criterion}
Let \(y_1,\ldots,y_s\) be self-adjoint elements of a tracial
\(C^*\)-probability space \((\mathcal A,\tau)\) with faithful trace, and let
\(Y_{1,n},\ldots,Y_{s,n}\in\M_{N_n}(\C)_{\rm sa}\).  The tuple may contain
both random coordinates and arbitrary deterministic companions.  Suppose
that, on the outcome under consideration, the following statements hold for
every fixed \(q\geq1\) and every self-adjoint coefficient list \(A\):
\begin{enumerate}[label=\textup{(\roman*)}]
 \item for every \(\varepsilon>0\), eventually
 \begin{equation}
  \Spec L_n(A)\subseteq
  \bigl\{t\in\R:\dist(t,\Spec L(A))<\varepsilon\bigr\};
  \label{eq:pencil-spectral-inclusion}
 \end{equation}
 \item for every integer \(p\geq1\),
 \begin{equation}
  (\tr_q\otimes\tr_{N_n})L_n(A)^{2p}
  \longrightarrow
  (\tr_q\otimes\tau)L(A)^{2p}.
  \label{eq:pencil-even-moment-convergence}
 \end{equation}
\end{enumerate}
Then \(Y_n\) converges strongly in joint \(*\)-distribution to \(y\).
Moreover, the conclusion holds at every separately fixed matrix
amplification.
\end{proposition}

This criterion is target independent: no freeness, Gaussianity, or
semicircular structure is assumed for \(y\).  Its proof is recalled in
Appendix~\ref{app:amplification-linearization}.  The main points are worth
recording here.  Passing to the quotient
\[
 \prod_n\M_{N_n}(\C)\big/\bigoplus_n\M_{N_n}(\C)
\]
turns \eqref{eq:pencil-spectral-inclusion} into an exact spectral inclusion.
The self-adjoint linearization trick then gives, for every fixed
matrix-valued \(*\)-polynomial \(P\),
\begin{equation}
 \limsup_{n\to\infty}\norm{P(Y_n)}\leq\norm{P(y)}.
 \label{eq:pencil-upper-bound}
\end{equation}
This use of linearization depends only on the polynomial and its spectral
parameter, not on the law of the target tuple.

Next, scalar shifts and matrix-unit polarization recover the normalized
trace of every noncommutative word from
\eqref{eq:pencil-even-moment-convergence}.  Thus, for every fixed
matrix-valued \(P\),
\begin{equation}
 (\tr_q\otimes\tr_{N_n})P(Y_n)
 \longrightarrow(\tr_q\otimes\tau)P(y).
 \label{eq:polynomial-trace-from-pencils}
\end{equation}
Finally, apply this convergence to \((P^*P)^p\).  With
\(a=P(y)^*P(y)\),
\begin{equation}
 \liminf_{n\to\infty}\norm{P(Y_n)}
 \geq\bigl[(\tr_q\otimes\tau)(a^p)\bigr]^{1/(2p)}.
 \label{eq:faithful-trace-lower-bound}
\end{equation}
Faithfulness implies that the right-hand side increases to \(\norm{P(y)}\)
as \(p\to\infty\).  Together with
\eqref{eq:pencil-upper-bound}, this proves norm convergence.  In the
probabilistic arguments below, comparison supplies (i)--(ii); the
linearization, polarization, and faithful-trace lower bound are the separate
deterministic mechanism that upgrades those inputs to the theorem.

\section{The covariance-matched Gaussian comparator}
\label{sec:gaussian-comparator}

This section proves the Gaussian identification asserted in
\Cref{prop:limit-identification}.  Two points require some care.  First,
covariance is always taken on the underlying real vector space, so that the
complex pseudocovariances are not discarded.  Second, the real Ginibre
matrices which realize the directed type blocks are obtained as genuine
corners of GOE matrices of twice the size.  This permits a direct application
of the scalar strong-asymptotic-freeness theorem of Fan--Sun--Wang
\cite[Theorem~4.3]{FanSunWang2021}.  Fixed matrix amplification is supplied
afterwards by \Cref{lem:fixed-amplification}, not by the cited theorem.

\subsection{Exact covariance factorizations}
\label{subsec:exact-gaussian-factorization}

For a complex vector \(z=(z_1,\ldots,z_k)\), write
\[
 \mathfrak r(z)=(\Rea z_1,\Ima z_1,\ldots,
                 \Rea z_k,\Ima z_k)\in\R^{2k}.
\]
All Gaussian factorizations below are factorizations of the covariance of
this realification.

Fix \(a<b\).  Orient an edge joining types \(a\) and \(b\) from \(a\) to
\(b\), irrespective of the numerical order of its two vertices.  Let
\((Z^{ab}_{+},Z^{ab}_{-})\) be the centered real Gaussian vector having the
same real covariance as an atom with law \(\mu^{ab}\).  As in
\eqref{eq:cross-real-factorization}, choose
\begin{equation}
 (Z^{ab}_{+,\kappa},Z^{ab}_{-,\kappa})_{\kappa\le d}
 =\sum_{\ell=1}^{L_{ab}}
   (\alpha^{ab}_{\kappa\ell},\beta^{ab}_{\kappa\ell})_{\kappa\le d}
   g^{ab}_{\ell},
 \qquad L_{ab}\le 4d,
 \label{eq:cross-factorization-recalled}
\end{equation}
where the \(g^{ab}_{\ell}\) are independent standard real Gaussians and the
coefficients may be complex.  Let
\((G^{ab}_{\ell,n})_{\ell\le L_{ab}}\) be independent real Ginibre matrices,
with all entries independent and distributed as \(N(0,n^{-1})\), and set
\begin{equation}
 Y^{ab}_{\kappa,n}
 =\sum_{\ell=1}^{L_{ab}}
  \left(
   \alpha^{ab}_{\kappa\ell}P_{a,n}G^{ab}_{\ell,n}P_{b,n}
   +\beta^{ab}_{\kappa\ell}P_{b,n}(G^{ab}_{\ell,n})^*P_{a,n}
  \right).
 \label{eq:cross-gaussian-matrix}
\end{equation}
For an edge with endpoint types \(a,b\), the standardized entries in the two
directions are the two components in
\eqref{eq:cross-factorization-recalled}, driven by the same real Gaussian
vector.  Thus \eqref{eq:cross-gaussian-matrix} retains the entire real
covariance between colors and orientations.  If the vertex of type \(b\)
has the smaller numerical label, the convention
\(\mu^{ba}=S_{\#}\mu^{ab}\) gives precisely the same assertion after
reversal.  Distinct unordered type pairs are assigned independent Ginibre
families.

For a same-type law \(\mu^{aa}\), let
\((Z^a_+,Z^a_-)\) denote its covariance-matched Gaussian pair and recall
\[
 U^a=\frac{Z^a_++Z^a_-}{\sqrt2},
 \qquad
 V^a=\frac{Z^a_+-Z^a_-}{\sqrt2}.
\]
Endpoint exchangeability sends \((U^a,V^a)\) to \((U^a,-V^a)\).  Every
real cross-covariance between the two centered vectors is consequently zero;
as they are jointly Gaussian, \(U^a\) and \(V^a\) are independent.  Choose
the factorizations from \eqref{eq:same-real-factorization},
\begin{equation}
 U^a_\kappa=\sum_{\ell=1}^{L_{a,+}}
                 \alpha^a_{\kappa\ell}g^a_{+,\ell},
 \qquad
 V^a_\kappa=\sum_{\ell=1}^{L_{a,-}}
                 \beta^a_{\kappa\ell}g^a_{-,\ell},
 \qquad L_{a,+},L_{a,-}\le 2d,
 \label{eq:same-factorization-recalled}
\end{equation}
using independent real Gaussian source families.  Take independent real
Ginibre matrices \(G^a_{+,\ell,n}\) and \(G^a_{-,\ell,n}\), independent also
of all the cross-type sources, and put
\begin{equation}
 S^a_{\ell,n}
 =\frac{G^a_{+,\ell,n}+(G^a_{+,\ell,n})^*}{\sqrt2},
 \qquad
 K^a_{\ell,n}
 =\frac{G^a_{-,\ell,n}-(G^a_{-,\ell,n})^*}{\sqrt2}.
 \label{eq:symmetric-skew-ginibre}
\end{equation}
The same-type polynomial model is
\begin{equation}
 \widetilde Y^{aa}_{\kappa,n}
 =\frac1{\sqrt2}P_{a,n}
  \left(
   \sum_{\ell=1}^{L_{a,+}}\alpha^a_{\kappa\ell}S^a_{\ell,n}
   +\sum_{\ell=1}^{L_{a,-}}\beta^a_{\kappa\ell}K^a_{\ell,n}
  \right)P_{a,n}.
 \label{eq:same-gaussian-matrix}
\end{equation}
For \(i<j\) of type \(a\), the variables
\(\sqrt n S^a_{\ell,n}(i,j)\) and
\(\sqrt n K^a_{\ell,n}(i,j)\) are independent standard real Gaussians,
while
\(S^a_{\ell,n}(j,i)=S^a_{\ell,n}(i,j)\) and
\(K^a_{\ell,n}(j,i)=-K^a_{\ell,n}(i,j)\).  Hence the standardized upper and
lower entries in \eqref{eq:same-gaussian-matrix} are respectively
\begin{equation}
 \frac{U^a_\kappa+V^a_\kappa}{\sqrt2}=Z^a_{+,\kappa},
 \qquad
 \frac{U^a_\kappa-V^a_\kappa}{\sqrt2}=Z^a_{-,\kappa}.
 \label{eq:same-entry-check}
\end{equation}
This proves exact agreement in law off the diagonal and also shows why the
outer factor \(2^{-1/2}\) in \eqref{eq:same-gaussian-matrix} is compulsory.

Let
\[
 \widetilde Y_{\kappa,n}
 =\sum_{a<b}Y^{ab}_{\kappa,n}
  +\sum_a\widetilde Y^{aa}_{\kappa,n}.
\]
If \(Y_{\kappa,n}\) denotes the full covariance-matched Gaussian comparator,
couple it with \(\widetilde Y_{\kappa,n}\) so that their off-diagonal entries
agree and write
\begin{equation}
 Y_{\kappa,n}=\widetilde Y_{\kappa,n}+D_{\kappa,n}.
 \label{eq:gaussian-diagonal-correction}
\end{equation}
Here \(D_{\kappa,n}\) is diagonal.  The prescribed diagonal covariance may
depend on the type; nevertheless the finite second-moment assumption and the
finiteness of the type set imply that every real coordinate of every diagonal
entry in \(D_{\kappa,n}\) has variance at most \(C/n\).  Therefore
\begin{equation}
 \mathbb P\{\norm{D_{\kappa,n}}>t\}
 \le Cn\exp(-cnt^2),
 \qquad t>0,
 \label{eq:diagonal-gaussian-tail}
\end{equation}
after adjusting constants.  It follows by Borel--Cantelli that
\begin{equation}
 \max_{\kappa\le d}\norm{D_{\kappa,n}}\longrightarrow0
 \quad\text{almost surely}.
 \label{eq:diagonal-gaussian-negligible}
\end{equation}
This correction is essential for an exact finite-\(n\) representation when
the comparator has zero or otherwise prescribed diagonal: the pure
Ginibre-polynomial model is exact off the diagonal, and the discrepancy on
the diagonal is norm-negligible.

\subsection{Real Ginibre matrices as doubled GOE corners}
\label{subsec:goe-doubling}

Let \(W_{\ell,2n}\) be a standard GOE matrix in the normalization of
Fan--Sun--Wang: its off-diagonal entries have variance \((2n)^{-1}\), and
its diagonal entries have variance \(n^{-1}\).  Write
\[
 e^{(n)}_{uv}=E_{uv}\otimes I_n\in\M_{2n}(\C),
 \qquad u,v\in\{1,2\}.
\]
In the corner \(e^{(n)}_{11}\M_{2n}e^{(n)}_{11}\cong\M_n\), define
\begin{equation}
 \Gamma_{\ell,n}
 =\sqrt2\,e^{(n)}_{11}W_{\ell,2n}
                  e^{(n)}_{22}e^{(n)}_{21}.
 \label{eq:ginibre-goe-corner}
\end{equation}
Under the corner identification, \(\Gamma_{\ell,n}\) is the \(n\)-by-\(n\)
matrix with entries
\(\sqrt2(W_{\ell,2n})_{i,n+j}\), \(1\le i,j\le n\).  These entries are
mutually independent \(N(0,n^{-1})\) variables.  In particular, the
positions \(i=j\) in this \(n\)-by-\(n\) matrix are still globally
off-diagonal GOE positions \((i,n+i)\); there is no exceptional diagonal in
the Ginibre corner.

For \(Z\in\M_n(\C)\), let \(\iota_n(Z)=E_{11}\otimes Z\), regarded as an
element of the \(e^{(n)}_{11}\)-corner.  Then
\begin{equation}
 \norm{\iota_n(Z)}=\norm{Z},
 \qquad
 \tr_{2n}(\iota_n(Z))=\frac12\tr_n(Z).
 \label{eq:corner-norm-trace}
\end{equation}
Thus the normalized trace on the corner is \(2\tr_{2n}\), while its operator
norm is the ordinary \(n\)-dimensional norm.

Set
\begin{equation}
 \widehat P_{a,n}=I_2\otimes P_{a,n},
 \qquad
 \widehat A_{j,n}=I_2\otimes A_{j,n},
 \label{eq:lifted-deterministic-variables}
\end{equation}
and consider the single deterministic tuple
\begin{equation}
 \mathcal D_n=
 \bigl((e^{(n)}_{uv})_{u,v\le2},
       (\widehat P_{a,n})_{a\le r},
       (\widehat A_{j,n})_{j\le m}\bigr).
 \label{eq:doubled-deterministic-tuple}
\end{equation}
The joint strong convergence of \((P_n,A_n)\), together with
\Cref{lem:fixed-amplification} at level \(2\), yields joint strong
convergence of \(\mathcal D_n\) to
\begin{equation}
 \mathcal D=
 \bigl((E_{uv}\otimes\one)_{u,v\le2},
       (I_2\otimes p_a)_{a\le r},
       (I_2\otimes a_j)_{j\le m}\bigr)
 \label{eq:doubled-deterministic-limit}
\end{equation}
in \((\M_2(\C)\otimes B,\tr_2\otimes\tau_B)\).  In particular, all
members of \(\mathcal D_n\) are uniformly bounded in norm.  When \(B\) is
presented only as a tracial \(C^*\)-algebra, we work in its faithful GNS
representation and then in the generated von Neumann algebra; faithfulness
preserves every operator norm under consideration.

Use one independent doubled GOE for every Ginibre source appearing in
\eqref{eq:cross-gaussian-matrix} and \eqref{eq:symmetric-skew-ginibre}.
The identities
\begin{align}
 \iota_n(P_{a,n}G_{\ell,n}P_{b,n})
 &=\widehat P_{a,n}\Gamma_{\ell,n}\widehat P_{b,n},
 \label{eq:masked-ginibre-polynomial}\\
 \iota_n(A_{j,n})
 &=e^{(n)}_{11}\widehat A_{j,n}e^{(n)}_{11}
 \label{eq:corner-companion-polynomial}
\end{align}
show that every masked cross-type block, its adjoint, every same-type
symmetric or skew block, and every companion is a fixed \(*\)-polynomial in
the doubled GOEs and \(\mathcal D_n\).

Fan--Sun--Wang's theorem applies at matrix size \(N=2n\): independent GOEs
with the normalization above, adjoined to a jointly strongly convergent
deterministic tuple, converge almost surely strongly to a free standard
semicircular family free from the deterministic limit
\cite[Theorem~4.3]{FanSunWang2021}.  Their theorem is stated for
deterministic matrices in \(\M_N(\C)\), not only real symmetric deterministic
matrices, so the matrix units and complex companions in \(\mathcal D_n\) are
admissible.  The conclusion is stated for self-adjoint polynomials.  For an
arbitrary scalar \(*\)-polynomial \(Q\),
the norm conclusion follows by applying it to \(Q^*Q\), while trace
convergence follows by applying it to \(\Rea Q\) and \(\Ima Q\).  This gives
scalar strong convergence of every polynomial model above.  The diagonal
perturbation \eqref{eq:diagonal-gaussian-negligible} does not alter that
conclusion.

\subsection{Compression and the explicit free limit}
\label{subsec:compressed-free-limit}

We identify the corner limit with the free variables in
\eqref{eq:cross-free-limit} and \eqref{eq:same-free-limit}.  Let
\((\mathcal N,\varphi)\) contain \(\M_2(\C)\otimes B\) and a free standard
semicircular family \((s_\ell)\), free from that algebra.  Put
\[
 e=E_{11}\otimes\one,
 \qquad f=E_{22}\otimes\one,
 \qquad v=E_{21}\otimes\one,
 \]
and equip \(e\mathcal Ne\) with its normalized trace
\(\varphi_e=2\varphi|_{e\mathcal Ne}\).  For every source define
\begin{equation}
 c_\ell=\sqrt2\,e s_\ell f v=\sqrt2\,e s_\ell v.
 \label{eq:circular-corner}
\end{equation}
The family \((c_\ell)\) is a standard free circular family, free from the
corner algebra \(e(\M_2\otimes B)e\cong B\).  Indeed, if \(E\) denotes the
conditional expectation onto \(\M_2\otimes B\), let
\(E_e=(eE(\,\cdot\,)e)|_{e\mathcal Ne}\), viewed through this corner
identification.  Then
\[
 E(s_\ell d s_k)=\delta_{\ell k}
  (\tr_2\otimes\tau_B)(d)\one.
\]
After compression, the only nonzero second cumulants are
\begin{equation}
 E_e(c_\ell b c_k^*)=E_e(c_\ell^*bc_k)
  =\delta_{\ell k}\tau_B(b)e,
 \qquad b\in B,
 \label{eq:circular-corner-covariance}
\end{equation}
and all higher cumulants vanish.  This proves both circularity and freeness
from \(B\), with
\(\varphi_e(c_\ell c_\ell^*)=\varphi_e(c_\ell^*c_\ell)=1\).

The limit of \eqref{eq:cross-gaussian-matrix} is therefore
\[
 x^{ab}_\kappa
 =\sum_{\ell=1}^{L_{ab}}
 \left(
  \alpha^{ab}_{\kappa\ell}p_ac^{ab}_\ell p_b
  +\beta^{ab}_{\kappa\ell}p_b(c^{ab}_\ell)^*p_a
 \right),
\]
which is \eqref{eq:cross-free-limit}.  For the two same-type source
families, write
\begin{equation}
 s^a_{+,\ell}=\frac{c^a_{+,\ell}+(c^a_{+,\ell})^*}{\sqrt2},
 \qquad
 s^a_{-,\ell}=\frac{i(c^a_{-,\ell}-(c^a_{-,\ell})^*)}{\sqrt2}.
 \label{eq:semicircular-from-circular}
\end{equation}
They form a free standard semicircular family, free from \(B\) and all the
other sources.  Since the skew matrix in
\eqref{eq:symmetric-skew-ginibre} converges to
\((c-c^*)/\sqrt2=-is_-\), the limit of
\eqref{eq:same-gaussian-matrix} is
\[
 x^{aa}_\kappa
 =\frac1{\sqrt2}p_a
  \left(
   \sum_{\ell=1}^{L_{a,+}}\alpha^a_{\kappa\ell}s^a_{+,\ell}
   -i\sum_{\ell=1}^{L_{a,-}}\beta^a_{\kappa\ell}s^a_{-,\ell}
  \right)p_a,
\]
exactly as in \eqref{eq:same-free-limit}.  Sources belonging to distinct
type pairs and sectors are free because their doubled GOEs were independent.

\begin{lemma}[Factorization invariance]
\label{lem:gaussian-factorization-invariance}
The joint \(*\)-distribution of the family
\(x_\kappa=\sum_{a<b}x^{ab}_\kappa+\sum_a x^{aa}_\kappa\), together with
\(B\), depends only on the real covariance matrices of the type-pair laws,
and not on any of the factorizations chosen above.
\end{lemma}

\begin{proof}
After zero columns have been added, two real factorizations of the same
positive semidefinite covariance matrix differ by an orthogonal change of
the standard real Gaussian sources.  A free circular family is invariant in
joint \(*\)-distribution under real orthogonal rotations, as is a free
semicircular family.  Apply this observation to each cross-type
factorization and separately to the two orthogonal same-type covariance
sectors.  Equivalently, both choices give the same finite-dimensional
Gaussian comparator for every \(n\), and the strong limit identified above
is unique.  The reduced free-product trace is faithful, as is its fixed
matrix amplification; hence, for every \(T\in\M_q(C^*(B,x))\),
\[
 \norm{T}^{2}=\lim_{k\to\infty}
 \bigl[(\tr_q\otimes\tau_B)((T^*T)^k)\bigr]^{1/k}.
\]
Thus equality of the amplified joint \(*\)-moments also gives equality of
all amplified polynomial norms.
\end{proof}

The covariance and normalization calculations underlying these formulas,
including the absence of any additional factor involving the type
proportions \(\pi_a\), are recorded in
Appendix~\ref{app:covariance-bookkeeping}.

\begin{proposition}[Strong limit of the Gaussian comparator]
\label{prop:gaussian-strong-limit}
On a single probability-one event,
\[
 (P_{1,n},\ldots,P_{r,n},A_n,Y_{1,n},\ldots,Y_{d,n})
 \]
converges jointly strongly in \(*\)-distribution to
\[
 (p_1,\ldots,p_r,a,x_1,\ldots,x_d).
\]
Consequently, for every separately fixed \(q\ge1\) and every fixed
\(Q\in\M_q(\C\langle z,z^*\rangle)\), both the normalized trace and the
operator norm of \(Q(P_n,A_n,Y_n)\) converge to those of
\(Q(p,a,x)\).  No uniformity in \(q\) is asserted.
\end{proposition}

\begin{proof}
The scalar assertion follows from the doubled-GOE argument and the
norm-negligible diagonal correction.  Apply
\Cref{lem:fixed-amplification} separately at every fixed \(q\).  A countable
intersection over \(q\), polynomials with rational real and imaginary
coefficients, and their adjoints gives one full-probability event.  Density
and the almost-sure uniform norm bounds extend the assertion to every fixed
complex coefficient list.
\end{proof}

We finally record the uniform integrability needed when covariance
comparison is formulated in terms of expected trace moments.

\begin{lemma}[Gaussian pencil moments]
\label{lem:gaussian-pencil-ui}
Let \(\cG_n\) be any fixed amplified self-adjoint pencil that is affine in
\(Y_n\), whose random coefficients may be left--right masked by the type
projections \(P_n\), and whose deterministic shift is a fixed affine pencil in
\((P_n,A_n)\).  Let \(\cG\) be the corresponding pencil in \((p,a,x)\).  For
every fixed integer \(k\ge1\),
\begin{equation}
 \sup_n\E\!\left[
  \left|(\tr_q\otimes\tr_n)(\cG_n^{2k})\right|^2
 \right]<\infty,
 \qquad
 \E(\tr_q\otimes\tr_n)(\cG_n^{2k})
 \longrightarrow(\tr_q\otimes\tau_B)(\cG^{2k}).
 \label{eq:gaussian-expected-moments}
\end{equation}
\end{lemma}

\begin{proof}
The pencil \(\cG_n\) is, up to the diagonal correction, an affine
combination of a fixed finite number of masked Ginibre matrices and uniformly
bounded deterministic matrices.  Through \eqref{eq:ginibre-goe-corner},
\[
 \norm{\cG_n}\le C\left(1+\sum_{\ell=1}^{L}\norm{W_{\ell,2n}}
                  +\max_{\kappa\le d}\norm{D_{\kappa,n}}\right)
\]
for constants \(C,L\) independent of \(n\).  Indeed, the GOE operator norm
is \(O(n^{-1/2})\)-Lipschitz in its standard Gaussian coordinates, while a
fixed-net bound for the quadratic forms \(u^*W_{\ell,2n}u\) gives
\(\sup_n\E\norm{W_{\ell,2n}}<\infty\).  Gaussian concentration and tail
integration therefore yield
\(\sup_n\E\norm{W_{\ell,2n}}^p<\infty\) for every fixed \(p\).  The same
bound for the diagonal correction follows directly from
\eqref{eq:diagonal-gaussian-tail}.  Since a normalized trace is bounded by
the operator norm, the first assertion in
\eqref{eq:gaussian-expected-moments} follows.  The almost-sure trace
convergence in \Cref{prop:gaussian-strong-limit}, together with this uniform
integrability, proves the second assertion.
\end{proof}

\section{Universality for bounded pair atoms}
\label{sec:bounded-transfer}

We now transfer the Gaussian limit of \cref{sec:gaussian-comparator} to
bounded unordered-pair atoms.  The comparison is performed separately for
each fixed matrix amplification.  Its essential feature is that all colors
and both orientations belonging to one unordered pair are kept inside a
single self-adjoint summand.  Thus no independence is imposed among the
coordinates of a pair atom.  The deterministic companions enter only through
the mean of the pencil, while the type projections simply select which of the
finitely many pair laws and coefficient blocks occurs.

The diagonal of the model is set equal to zero in the transfer theorem below.
This is the form needed after radial truncation in \cref{sec:exact-l4}.  We
nevertheless include a bounded diagonal in the parameter calculation, both
for completeness and to make clear why the merely \(L^2\) diagonal of the
main theorem must be removed before the present argument.

\subsection{One self-adjoint summand per unordered pair}
\label{subsec:bounded-pair-summands}

Fix \(q\geq1\), matrices \(B_\kappa^{ab}\in\M_q(\C)\), and a deterministic
self-adjoint matrix \(C_n\in\M_q\otimes\M_n\).  The latter may be any fixed
self-adjoint affine pencil in \((P_n,A_n)\); it need not be sparse or structured.
Consider
\begin{equation}
 \begin{split}
 \cL_n={}&C_n+
 \sum_{a,b=1}^r\sum_{\kappa=1}^d
 \bigl(
 B_\kappa^{ab}\otimes P_{a,n}X_{\kappa,n}P_{b,n}\bigr.\\
 &\hspace{39mm}\bigl.
 +(B_\kappa^{ab})^*\otimes
 P_{b,n}X_{\kappa,n}^*P_{a,n}
 \bigr).
 \end{split}
 \label{eq:typed-affine-pencil}
\end{equation}
The usual unmasked random pencil is obtained by taking
\(B_\kappa^{ab}=B_\kappa\) for all \((a,b)\).  Type-masked pencils are useful
because their Gaussian limits can be read directly from the corners of
\eqref{eq:total-free-limit}.

Let \(i<j\), put \(a=\theta(i)\), \(b=\theta(j)\), and define
\begin{equation}
 U_{ij}:=\sum_{\kappa=1}^d
 \left(
 B_\kappa^{ab}\xi_{ij}^{(\kappa)}
 +(B_\kappa^{ba})^*\overline{\xi_{ji}^{(\kappa)}}
 \right).
 \label{eq:typed-Uij}
\end{equation}
The complete contribution of \(V_{ij}\) to \eqref{eq:typed-affine-pencil}
is
\begin{equation}
 Z_{ij}^{(n)}=\frac1{\sqrt n}
 \left(U_{ij}\otimes E_{ij}+U_{ij}^*\otimes E_{ji}\right).
 \label{eq:typed-pair-summand}
\end{equation}
In particular, \(Z_{ij}^{(n)}\) is centered and self-adjoint, and the family
\((Z_{ij}^{(n)})_{i<j}\) is independent.  Splitting the two orientations into
different summands would be incorrect: their dependence is part of the law of
\(V_{ij}\).

If a bounded diagonal is retained, write \(a=\theta(i)\) and set
\begin{equation}
 H_i:=\sum_{\kappa=1}^d
 \left(B_\kappa^{aa}\xi_{ii}^{(\kappa)}
 +(B_\kappa^{aa})^*\overline{\xi_{ii}^{(\kappa)}}\right),
 \qquad
 Z_{ii}^{(n)}=\frac1{\sqrt n}H_i\otimes E_{ii}.
 \label{eq:typed-diagonal-summand}
\end{equation}
Then \(H_i=H_i^*\), and mutual independence of the diagonal atoms makes
these additional independent centered summands.  Thus, with the diagonal
included when appropriate,
\begin{equation}
 \cL_n=C_n+\sum_{i<j}Z_{ij}^{(n)}+\sum_iZ_{ii}^{(n)}.
 \label{eq:bvh-summand-model}
\end{equation}
This is exactly the independent-summand model of
\cite[Section~2.1]{BvH2024}; identical distributions are not required.

\subsection{The three comparison parameters}

For the fixed coefficient list, define
\begin{equation}
 \beta_{ab}^2:=\sum_{\kappa=1}^d
 \left(\norm{B_\kappa^{ab}}^2+\norm{B_\kappa^{ba}}^2\right),
 \qquad \beta:=\max_{a,b}\beta_{ab},
 \label{eq:beta-off-diagonal}
\end{equation}
and
\begin{equation}
 \beta_\Delta^2:=4\max_a\sum_{\kappa=1}^d
 \norm{B_\kappa^{aa}}^2.
 \label{eq:beta-diagonal}
\end{equation}
The factor \(4\) in \eqref{eq:beta-diagonal} accounts for the two terms in
each diagonal coordinate.

\begin{lemma}[Pair-pencil parameters]
\label{lem:bounded-pair-parameters}
Suppose that \(\norm{V_{ij}}_2\leq K\) almost surely and, when the diagonal
is retained, \(\norm{W_i}_2\leq K_\Delta\) almost surely.  For every fixed
pencil \eqref{eq:typed-affine-pencil}, the parameters of
\cite[Section~2.1]{BvH2024} satisfy
\begin{equation}
 R(\cL_n)=O(n^{-1/2}),\qquad
 \sigma(\cL_n)=O(1),\qquad
 \sigma_*(\cL_n)=O(n^{-1/2}).
 \label{eq:three-bvh-scales}
\end{equation}
The implied constants may depend on the fixed pencil, \(K\),
\(K_\Delta\), and the finite family of type laws, but not on \(n\).
They are unchanged if an arbitrary deterministic self-adjoint \(C_n\) is
added.
\end{lemma}

\begin{proof}
Cauchy--Schwarz in \eqref{eq:typed-Uij} gives
\begin{equation}
 \norm{U_{ij}}\leq\beta_{ab}\norm{V_{ij}}_2.
 \label{eq:Uij-bound}
\end{equation}
On the \((i,j)\)-coordinate subspace, \eqref{eq:typed-pair-summand} is
\(n^{-1/2}\begin{psmallmatrix}0&U_{ij}\\U_{ij}^*&0\end{psmallmatrix}\),
whose norm is \(n^{-1/2}\norm{U_{ij}}\).  Similarly,
\(\norm{H_i}\leq\beta_\Delta\norm{W_i}_2\).  Since \(R\) is the essential
supremum of the \emph{maximum} summand norm, rather than the sum of those
norms,
\begin{equation}
 R(\cL_n)\leq
 \frac{\max\{\beta K,\beta_\Delta K_\Delta\}}{\sqrt n}.
 \label{eq:R-bound}
\end{equation}
For the centered radial truncation \eqref{eq:radial-truncation}, the safe
bound is \(\norm{V^{(K),ab}}_2\leq2K\); this only changes the constant in
\eqref{eq:R-bound}.

Put
\begin{equation}
 C_U:=\max_{a,b}\E\norm{U^{ab}}^2,
 \qquad C_\Delta:=\max_a\E\norm{H^a}^2.
 \label{eq:second-moment-pencil-constants}
\end{equation}
Both constants are finite and depend only on the fixed coefficient list and
the second moments of the finite family of local laws.  Squaring one pair
summand gives the exact identity
\begin{equation}
 (Z_{ij}^{(n)})^2=\frac1n
 \left(U_{ij}U_{ij}^*\otimes E_{ii}
       +U_{ij}^*U_{ij}\otimes E_{jj}\right).
 \label{eq:pair-summand-square}
\end{equation}
Consequently, the variance matrix is block diagonal in the vertex index.  If
\(\Gamma_{i,n}\in\M_q\) denotes its \(i\)-th block, then
\begin{equation}
 \Gamma_{i,n}=\frac1n\left(
 \sum_{j>i}\E U_{ij}U_{ij}^*
 +\sum_{j<i}\E U_{ji}^*U_{ji}+\E H_i^2\right),
 \label{eq:variance-block}
\end{equation}
with the last term omitted when the diagonal is deleted.  It follows that
\begin{equation}
 \sigma(\cL_n)^2=\max_i\norm{\Gamma_{i,n}}
 \leq C_U+C_\Delta/n.
 \label{eq:sigma-bound}
\end{equation}

Finally, write unit vectors in \(\C^q\otimes\C^n\) as
\(u=(u_1,\ldots,u_n)\) and \(v=(v_1,\ldots,v_n)\).  Then
\begin{equation}
 \ip{u}{Z_{ij}^{(n)}v}=\frac1{\sqrt n}
 \left(\ip{u_i}{U_{ij}v_j}+\ip{u_j}{U_{ij}^*v_i}\right).
 \label{eq:weak-variance-bilinear}
\end{equation}
Using \(\abs{s+t}^2\leq2\abs{s}^2+2\abs{t}^2\), summing over pairs, and
using \(\sum_i\norm{u_i}^2=\sum_i\norm{v_i}^2=1\), we obtain
\begin{equation}
 \sum_{i<j}\E\abs{\ip{u}{Z_{ij}^{(n)}v}}^2
 +\sum_i\E\abs{\ip{u}{Z_{ii}^{(n)}v}}^2
 \leq\frac{2C_U+C_\Delta}{n}.
 \label{eq:sigma-star-bound}
\end{equation}
Independence and centering identify the left-hand side with the variance of
the bilinear form of the full centered pencil.  Taking the supremum over
\(u,v\) proves the asserted bound on \(\sigma_*\).  The deterministic matrix
\(C_n\) contributes to none of the three parameters.
\end{proof}

\subsection{Quantitative comparison with the Gaussian pencil}

Let \(\cG_n\) be obtained by replacing, independently for every unordered
pair, the realification of \(V_{ij}\) by a centered Gaussian vector with the
same full real covariance and then applying the real-linear map
\eqref{eq:typed-Uij}.  Retained diagonal atoms are treated analogously, and
the deterministic shift \(C_n\) is left unchanged.  The resulting
self-adjoint matrix has the same mean and full complex entry covariance as
\(\cL_n\); hence it is precisely the Gaussian comparator of
\cite[Section~2.1]{BvH2024}.  In particular, all correlations between colors
and orientations are preserved.

Let
\begin{equation}
 D_n:=qn
 \label{eq:ambient-pencil-dimension}
\end{equation}
be the exact ambient matrix dimension.  To avoid conflict with this external
amplification, we denote the moment-integrability index in
\cite[Theorem~2.9]{BvH2024} by \(q_{\rm mom}\).

\begin{proposition}[Comparison rates for a bounded pencil]
\label{prop:bounded-comparison-rates}
For every fixed pencil as above, there are constants independent of \(n\)
such that the following assertions hold.

First, for every \(t\geq0\),
\begin{equation}
 \mathbb P\!\left\{
 d_{\rm H}(\Spec\cL_n,\Spec\cG_n)>C\varepsilon_n(t)
 \right\}\leq D_ne^{-t},
 \label{eq:bvh-spectrum-comparison}
\end{equation}
where \(C\) is universal and
\begin{equation}
 \varepsilon_n(t)=\sigma_*(\cL_n)t^{1/2}
 +R(\cL_n)^{1/3}\sigma(\cL_n)^{2/3}t^{2/3}
 +R(\cL_n)t.
 \label{eq:bvh-spectrum-error}
\end{equation}
With \(t=4\log D_n\),
\begin{equation}
 \varepsilon_n(4\log D_n)
 =O\!\left(n^{-1/2}(\log n)^{1/2}
 +n^{-1/6}(\log n)^{2/3}+n^{-1/2}\log n\right)=o(1),
 \label{eq:bvh-spectrum-rate}
\end{equation}
and the failure probability in \eqref{eq:bvh-spectrum-comparison} is exactly
bounded by \(D_n^{-3}\).

Second, for every fixed \(p\geq1\), choosing
\(q_{\rm mom}=\infty\) in \cite[Theorem~2.9]{BvH2024} gives
\begin{equation}
 \left|
 \bigl(\E\tr_{D_n}\cL_n^{2p}\bigr)^{1/(2p)}
 -\bigl(\E\tr_{D_n}\cG_n^{2p}\bigr)^{1/(2p)}
 \right|
 =O_p(n^{-1/6}+n^{-1/2}).
 \label{eq:bvh-expected-moment-rate}
\end{equation}
Here \(\tr_{D_n}=D_n^{-1}\Tr\), with no additional amplification factor.

Finally, put
\begin{equation}
 a_{n,p}:=\bigl(\E\tr_{D_n}\cL_n^{2p}\bigr)^{1/(2p)}.
 \label{eq:original-expected-moment-root}
\end{equation}
If \(\sup_na_{n,p}<\infty\), then, for all sufficiently large \(n\),
\begin{multline}
 \mathbb P\!\left\{
 \left|\bigl(\tr_{D_n}\cL_n^{2p}\bigr)^{1/(2p)}-a_{n,p}\right|
 >C_p\left[n^{-1/4}(\log n)^{1/2}+n^{-1/2}\log n\right]
 \right\}\\
 \leq2D_n^{-4}.
 \label{eq:bvh-moment-concentration-rate}
\end{multline}
\end{proposition}

\begin{proof}
Equation \eqref{eq:bvh-spectrum-comparison} is
\cite[Theorem~2.6]{BvH2024}.  Substitution of
\eqref{eq:three-bvh-scales} and \(t=4\log D_n\) gives
\eqref{eq:bvh-spectrum-rate} and
\(D_ne^{-4\log D_n}=D_n^{-3}\).

The convention of \cite[Section~2.1]{BvH2024} is
\(R_\infty=R\) and \(\sigma_\infty=\sigma\).  Its moment comparison
therefore bounds the left-hand side of
\eqref{eq:bvh-expected-moment-rate} by a constant multiple of
\[
 R(\cL_n)^{1/3}\sigma(\cL_n)^{2/3}p^{2/3}
 +R(\cL_n)p,
\]
which is \(O_p(n^{-1/6}+n^{-1/2})\).

For the last assertion, \cite[Lemma~9.20]{BvH2024} states, for \(t\geq p\),
that the deviation in \eqref{eq:bvh-moment-concentration-rate} is at most
\begin{equation}
 C\left(\sigma_*(\cL_n)
 +R(\cL_n)^{1/2}a_{n,p}^{1/2}\right)t^{1/2}
 +CR(\cL_n)t
 \label{eq:bvh-concentration-threshold}
\end{equation}
outside an event of probability \(2e^{-t}\).  Taking
\(t=4\log D_n\), which is eventually at least \(p\), and using the assumed
bound on \(a_{n,p}\) gives the displayed rate and
\(2e^{-4\log D_n}=2D_n^{-4}\).
\end{proof}

The three parts of \cref{prop:bounded-comparison-rates} serve different
purposes.  Spectral comparison supplies the no-outlier statement needed for
linearization.  Expected-moment comparison identifies the deterministic
moment limit.  Concentration then replaces expected moments by the random
normalized moments.  In particular, the last step cannot be invoked until
the expected roots \(a_{n,p}\) have been shown uniformly bounded.

\subsection{Transfer of the Gaussian strong limit}

Suppose that \(C_n\) is the evaluation of a fixed self-adjoint affine pencil
in \((P_n,A_n)\), and denote its limit by \(C\in\M_q\otimes B\).  The target
of \eqref{eq:typed-affine-pencil} is
\begin{equation}
 \begin{split}
 \cL={}&C+
 \sum_{a,b=1}^r\sum_{\kappa=1}^d
 \bigl(
 B_\kappa^{ab}\otimes p_ax_\kappa p_b\bigr.\\
 &\hspace{39mm}\bigl.
 +(B_\kappa^{ab})^*\otimes p_bx_\kappa^*p_a
 \bigr).
 \end{split}
 \label{eq:typed-limit-pencil}
\end{equation}

\begin{proposition}[Bounded affine-pencil transfer]
\label{prop:bounded-pencil-transfer}
Assume that the finite family of off-diagonal pair laws has bounded support
and that the random diagonal is zero.  In the nested model, on one
probability-one event, simultaneously for every fixed \(q\), every fixed
self-adjoint coefficient list in \eqref{eq:typed-affine-pencil}, and every
\(p\geq1\),
\begin{align}
 d_{\rm H}(\Spec\cL_n,\Spec\cL)&\longrightarrow0,
 \label{eq:bounded-pencil-spectrum-limit}\\
 \tr_{D_n}\cL_n^{2p}
 &\longrightarrow(\tr_q\otimes\tau_B)(\cL^{2p}).
 \label{eq:bounded-pencil-moment-limit}
\end{align}
For a fixed-law nonnested triangular array, the same two convergences hold in
probability for every separately fixed pencil and \(p\).
\end{proposition}

\begin{proof}
Fix first \(q\), \(p\), and a coefficient list whose real and imaginary
parts are rational.  By \cref{prop:gaussian-strong-limit}, the
covariance-matched Gaussian pencil \(\cG_n\) converges strongly almost surely
to \(\cL\).  In particular, its spectrum converges in Hausdorff distance to
\(\Spec\cL\).  This standard consequence of strong convergence follows from
polynomial approximation of the resolvent off \(\Spec\cL\) for the upper
inclusion and from trace convergence together with faithfulness for the lower
inclusion.  Combining this fact with
\eqref{eq:bvh-spectrum-comparison}--\eqref{eq:bvh-spectrum-rate} and
Borel--Cantelli proves \eqref{eq:bounded-pencil-spectrum-limit}, since
\(\sum_nD_n^{-3}<\infty\).  The spectrum comparison theorem is valid for
every coupling of \(\cL_n\) and \(\cG_n\), so we may use the nested Gaussian
realization from \cref{sec:gaussian-comparator} on a product space.  The
resulting statement depends only on the original pencil, and Fubini returns it
to the original probability space.

Strong convergence of the Gaussian tuple alone is not used to interchange
expectation and limit.  Instead, \cref{lem:gaussian-pencil-ui} gives
\begin{equation}
 \E\tr_{D_n}\cG_n^{2p}
 \longrightarrow(\tr_q\otimes\tau_B)(\cL^{2p}).
 \label{eq:gaussian-expected-moment-limit}
\end{equation}
The comparison \eqref{eq:bvh-expected-moment-rate} then identifies the limit
of \(a_{n,p}\) and, in particular, gives \(\sup_na_{n,p}<\infty\) after
discarding finitely many \(n\).  We may therefore apply
\eqref{eq:bvh-moment-concentration-rate}.  Its exceptional probabilities are
summable, and a second Borel--Cantelli argument proves
\eqref{eq:bounded-pencil-moment-limit}.

Intersect these events over the countable set of \(q,p\), and rational
self-adjoint coefficient lists.  Coordinate pencils and the Gaussian limit
give eventual boundedness of every random coordinate; joint strong
convergence gives uniform boundedness of the deterministic coordinates.
Rational approximation therefore extends the spectral assertion to arbitrary
fixed coefficient matrices.  The telescoping estimate
\begin{equation}
 \abs{\tr_D(S^{2p})-\tr_D(T^{2p})}
 \leq2p\max\{\norm S,\norm T\}^{2p-1}\norm{S-T}
 \label{eq:moment-telescoping}
\end{equation}
extends the moment assertion.  This produces one event for every fixed matrix
level, but no bound uniform in \(q\).

For a nonnested row, the same estimates are used without a countable
almost-sure intersection.  The failure probabilities in
\eqref{eq:bvh-spectrum-comparison} and
\eqref{eq:bvh-moment-concentration-rate} tend to zero, and the Gaussian
quantities have the same limiting law.  Hence each fixed pencil satisfies
\eqref{eq:bounded-pencil-spectrum-limit}--\eqref{eq:bounded-pencil-moment-limit}
in probability.
\end{proof}

\begin{theorem}[Bounded finite-type universality]
\label{thm:bounded-transfer}
Under the deterministic hypotheses of \cref{sec:model-results}, suppose that
the finite family of off-diagonal pair laws is centered, endpoint-consistent,
and has bounded support, and set the random diagonal equal to zero.  Then
\((P_n,A_n,X_n)\) converges jointly strongly to \((p,a,x)\), at every
separately fixed matrix amplification.  The convergence holds on one common
probability-one event in the nested model and in probability, separately for
each fixed matrix-valued polynomial, in the fixed-law nonnested model.
\end{theorem}

\begin{proof}
Replace every non-self-adjoint coordinate by its real and imaginary parts as
in \cref{subsec:self-adjoint-pencils}.  The ordinary affine pencils in this
self-adjoint tuple are the special case of \eqref{eq:typed-affine-pencil} with
\(B_\kappa^{ab}\) independent of \((a,b)\), while the deterministic part
contains the projection and companion coordinates.  Thus
\cref{prop:bounded-pencil-transfer} supplies the spectral-inclusion and
even-moment hypotheses of the companion-augmented, target-independent criterion
\cref{prop:affine-pencil-criterion}.  That criterion---not the
semicircular-only linearization statement in \cite{BvH2024}---yields all
polynomial norm and trace limits, including mixed polynomials in the
companions.

In the nonnested case, apply the standard subsequence principle.  From every
subsequence, the in-probability pencil limits permit a further subsequence on
which the countable rational pencil conditions hold almost surely.  Applying
\cref{prop:affine-pencil-criterion} on that subsequence gives the desired
limit for the fixed polynomial.  Hence the original sequence converges in
probability.  Finally, \cref{lem:fixed-amplification} supplies every fixed
matrix level; neither argument gives uniformity for \(q=q_n\to\infty\).
\end{proof}

\section{Removal of boundedness at the fourth-moment threshold}
\label{sec:exact-l4}

We now remove the bounded-support assumption.  The truncation must be
performed on an entire unordered-pair atom: a coordinatewise truncation need
not preserve either the correlations between colors or the endpoint-reversal
law.  The distinction between nested and nonnested models also first becomes
essential in this step.

\subsection{A continuity estimate}

We use the following elementary estimate repeatedly.  It applies equally to
scalar- and matrix-valued polynomials.

\begin{lemma}[Polynomial continuity]
\label{lem:polynomial-continuity}
Let \(Q=\sum_w C_w\otimes w\) be a fixed matrix-valued
\(*\)-polynomial.  Suppose that two tuples \(Y=(Y_j)_j\) and
\(Z=(Z_j)_j\), including the adjoint coordinates, lie in the operator-norm
ball of radius \(M\geq1\), and put

\[
 \varepsilon=\max_j\norm{Y_j-Z_j}.
\]
Then
\begin{equation}
 \norm{Q(Y)-Q(Z)}
 \leq L_{Q,M}\varepsilon,
 \qquad
 L_{Q,M}:=\sum_{|w|\geq1}\norm{C_w}|w|M^{|w|-1}.
 \label{eq:polynomial-continuity}
\end{equation}
The same right-hand side bounds the difference of the corresponding
normalized traces.
\end{lemma}

\begin{proof}
For a word \(w=y_1\cdots y_s\), insert and subtract one factor at a
time.  The resulting telescoping sum has \(s\) terms, each bounded by
\(M^{s-1}\varepsilon\).  Summing over the words proves the norm estimate;
the trace estimate follows from \(\abs{\tr(H)}\leq\norm H\).
\end{proof}

\subsection{Centered radial truncation}

For a type pair \((a,b)\) and \(K\in\N\), define
\begin{align}
 V^{(K),ab}
 &=V^{ab}\one_{\{\norm{V^{ab}}_2\leq K\}}
   -\E\!\left[V^{ab}\one_{\{\norm{V^{ab}}_2\leq K\}}\right],
 \label{eq:radial-truncation}\\
 R^{(K),ab}&=V^{ab}-V^{(K),ab}.
 \label{eq:radial-residual}
\end{align}
Both vectors are centered.  Since the Euclidean ball is invariant under the
endpoint swap \(S\), the identities

\[
 \mu^{ba}=S_{\#}\mu^{ab},\qquad \mu^{aa}=S_{\#}\mu^{aa}
\]
remain valid after truncation and for the residuals.  Set
\[
 \delta_K=\max_{a,b}
 \norm{R^{(K),ab}}_{L^2(\C^{2d})}.
\]
Centering is an orthogonal projection in \(L^2\), and the number of type
pairs is finite, so
\begin{equation}
 \delta_K^2
 \leq\max_{a,b}\E\!\left[
   \norm{V^{ab}}_2^2\one_{\{\norm{V^{ab}}_2>K\}}
 \right]\longrightarrow0.
 \label{eq:delta-K}
\end{equation}
Let \(X_{\kappa,n}^{(K)}\) be the matrix formed from the truncated atoms,
with zero diagonal.  The original diagonal will be restored below.

\subsection{Cross-type residual blocks}

We first record the rectangular input.  If \(z\) is a centered complex
variable with finite fourth moment and \(Z_n=(z_{ij})_{i,j\leq n}\) is the
\(n\)-th corner of one infinite iid array, the \(k=1\) case of
Bai--Yin's iid matrix norm estimate \cite[Theorem~2.1]{BaiYin1986},
applied separately to real and imaginary parts, gives
\begin{equation}
 \limsup_{n\to\infty}\norm{n^{-1/2}Z_n}
 \leq2\bigl(\norm{\Rea z}_2+\norm{\Ima z}_2\bigr)
 \leq2\sqrt2\norm z_2
 \quad\text{almost surely}.
 \label{eq:complex-bai-yin}
\end{equation}
No independence between the real and imaginary parts is used.

Fix \(a\ne b\), \(\kappa\leq d\), and \(K\).  Complete the observed
\(a\)-to-\(b\) residual entries on an auxiliary probability space to one
full infinite iid array \(Y^{ab,+,\kappa,K}\) with the appropriate scalar
marginal.  Then the normalized directed block is exactly
\[
 D_{\kappa,n}^{ab,K}
 =P_{a,n}\frac{Y_n^{ab,+,\kappa,K}}{\sqrt n}P_{b,n}.
\]
The reversal relation ensures that all entries in this ordered block have
the same law, irrespective of which endpoint has the smaller numerical
index.  Compression contractivity and \eqref{eq:complex-bai-yin} yield
\begin{equation}
 \limsup_n\norm{D_{\kappa,n}^{ab,K}}
 \leq2\sqrt2\,\delta_K
 \quad\text{almost surely}.
 \label{eq:cross-tail-one}
\end{equation}
The reverse direction is completed in a separate iid array.  The two
observed directions may be correlated, as prescribed by their common pair
atom; separate completions do not modify this correlation and avoid imposing
a false independence.  Taking a finite intersection and summing the directed
blocks gives
\begin{equation}
 \limsup_n
 \norm{\sum_{a\ne b}D_{\kappa,n}^{ab,K}}
 \leq2\sqrt2\,r(r-1)\delta_K
 \quad\text{almost surely}.
 \label{eq:cross-tail-all}
\end{equation}
The compressed quantities depend only on the original atoms.  Fubini's
theorem therefore removes the auxiliary fillers from
\eqref{eq:cross-tail-one}--\eqref{eq:cross-tail-all}.

\subsection{Same-type residual blocks}

The two orientations of a same-type atom require a different reduction.
Fix \(a,\kappa,K\), write

\[
 u=R^{(K),aa}_{+,\kappa},\qquad
 v=R^{(K),aa}_{-,\kappa},\qquad
 p=\frac{u+v}{\sqrt2},\quad q=\frac{u-v}{\sqrt2},
\]
and let \(M_{\kappa,n}^{aa,K}\) be the corresponding unnormalized
zero-diagonal residual matrix, extended by zero away from the type-\(a\)
block.  If \(S_n\) is complex symmetric with off-diagonal atom \(p\), and
\(A_n\) is complex skew-symmetric with upper-triangular atom \(q\), then
\begin{equation}
 M_{\kappa,n}^{aa,K}=\frac{S_n+A_n}{\sqrt2}.
 \label{eq:same-type-decomposition}
\end{equation}
Split \(S_n\) and \(A_n\) into real and imaginary parts.  This produces two
real symmetric and two real skew-symmetric arrays.  Complete each same-type
block once to an infinite upper-triangular iid array.  Anderson's
restatement of the Bai--Yin theorem \cite[Theorem~3]{Anderson2013}, applied
to both signs, controls the two symmetric arrays.  The explicitly twisted
version in \cite[Remark~4]{Anderson2013}, again applied to both signs,
controls \(iA_n\) and \(-iA_n\).  Remark~3 of that source includes the
variance-zero case.  The completed arrays have centered off-diagonal laws
with finite fourth moments and zero diagonal, so all the hypotheses are
satisfied even when some residual coordinate is degenerate.

Put
\[
 \rho_{a\kappa K}^2=\E|u|^2+\E|v|^2.
\]
The transform from \((u,v)\) to \((p,q)\) is orthogonal, and hence the sum
of the four real-coordinate standard deviations is at most
\(2\rho_{a\kappa K}\).  Compression, the triangle inequality, and
\eqref{eq:same-type-decomposition} therefore give the deliberately
nonoptimal estimate
\begin{equation}
 \limsup_n\norm{n^{-1/2}M_{\kappa,n}^{aa,K}}
 \leq2\sqrt2\,\rho_{a\kappa K}
 \leq2\sqrt2\,\delta_K
 \quad\text{almost surely}.
 \label{eq:same-tail-one}
\end{equation}
Reindexing the type-\(a\) subsequence would reveal a factor
\(\sqrt{\pi_a}\), but it is unnecessary here.  Distinct same-type blocks
are block diagonal, so their sum is controlled by the maximum of their
norms.  Again Fubini removes the auxiliary completion.

Combining \eqref{eq:cross-tail-all}, \eqref{eq:same-tail-one}, and
\eqref{eq:delta-K}, and intersecting over finitely many colors and types and
countably many integer \(K\), proves the following estimate.

\begin{proposition}[Nested residual estimate]
\label{prop:nested-residual}
For the nested infinite-array model,
\begin{equation}
 \lim_{K\to\infty}\limsup_{n\to\infty}
 \max_{\kappa\leq d}
 \norm{X_{\kappa,n}^{\circ}-X_{\kappa,n}^{(K)}}=0
 \quad\text{almost surely},
 \label{eq:nested-residual-limit}
\end{equation}
where \(X_{\kappa,n}^{\circ}\) denotes the original matrix with its
diagonal deleted.
\end{proposition}

For nonnested rows with the same finite family of laws, let
\(R_{n,K}\) denote the maximum residual norm on the left below, and construct
\(R^{\mathrm{ref}}_{n,K}\) from fixed nested infinite iid completions using
the same deterministic type sets.  For every \(n,K\) and \(\varepsilon>0\),
\[
 \mathbb P\{R_{n,K}>\varepsilon\}
 =\mathbb P\{R^{\mathrm{ref}}_{n,K}>\varepsilon\}.
\]
The almost-sure Bai--Yin--Anderson bounds for the reference arrays therefore
imply, after a finite union over types and colors,
\begin{equation}
 \lim_{K\to\infty}\limsup_{n\to\infty}
 \mathbb P\!\left\{
 \max_{\kappa\leq d}
 \norm{X_{\kappa,n}^{\circ}-X_{\kappa,n}^{(K)}}>\varepsilon
 \right\}=0.
 \label{eq:triangular-residual}
\end{equation}

\subsection{Diagonal entries}

Only a second moment is needed on the diagonal.  For a nested sequence whose
laws belong to a fixed finite type family, the tail-sum criterion gives
\[
 \sum_{i:\theta(i)=a}
 \mathbb P\{ |d_i|>\varepsilon\sqrt i\}<\infty.
\]
Borel--Cantelli, followed by a split into finitely many early indices and
the remaining indices, yields
\begin{equation}
 \frac1{\sqrt n}\max_{i\leq n}|d_i|\longrightarrow0
 \quad\text{almost surely}.
 \label{eq:nested-diagonal}
\end{equation}
For nonnested rows the union bound and
\(t^2\mathbb P\{|d|>t\}\to0\) under \(L^2\) give the same conclusion in
probability.  Thus deletion or restoration of the diagonal does not affect
any fixed polynomial test.

\subsection{Passage to the full covariance and proof of the main theorems}

Let \(x^{(K)}\) be the free Gaussian family obtained from the covariance of
the truncated laws.  Realify all type-pair atoms and realize their Gaussian
factors by applying the positive-semidefinite covariance square roots to
one fixed finite collection of standard free circular and semicircular
sources.  The \(L^2\)-convergence in \eqref{eq:delta-K} implies convergence
of these finite real covariance matrices.  Continuity of the
positive-semidefinite square root, including at a singular covariance,
then gives
\begin{equation}
 \max_{\kappa\leq d}\norm{x_\kappa^{(K)}-x_\kappa}
 \longrightarrow0.
 \label{eq:free-tail-continuity}
\end{equation}

For fixed \(K\), the bounded theorem of \cref{sec:bounded-transfer} applies
to \(X_n^{(K)}\).  In the nested model take a countable intersection over
\(K\), the rational matrix-valued polynomials, and the events in
\cref{prop:nested-residual}.  One fixed truncation level first gives an
eventual uniform norm bound on the original and truncated tuples.  For a
fixed polynomial \(Q\), the three differences
\[
 Q(X_n)-Q(X_n^{(K)}),\qquad
 Q(X_n^{(K)})-Q(x^{(K)}),\qquad
 Q(x^{(K)})-Q(x)
\]
are controlled, respectively, by \cref{lem:polynomial-continuity}, the
bounded theorem, and \eqref{eq:free-tail-continuity}.  First let
\(n\to\infty\) and then \(K\to\infty\).  This proves norm and normalized
trace convergence on one probability-one event.  Approximation of arbitrary
fixed coefficients by rational ones completes the countable-event argument.

For nonnested rows the same three-term argument uses
\eqref{eq:triangular-residual} and convergence in probability.  It is made
separately for each fixed coefficient size and each fixed polynomial; no
common almost-sure event and no uniformity in the coefficient size result.
This proves the two main convergence assertions stated in
\cref{sec:model-results}.

\subsection{Why the coupling distinction is sharp}

The fourth moment is necessary already on the real Wigner subfamily by the
necessary direction of Bai--Yin \cite{BaiYin1988}.  There is also no
coupling-invariant almost-sure theorem for fresh triangular rows at this
threshold.  Indeed, choose a centered symmetric law satisfying, for large
\(t\),
\[
 \mathbb P\{|\xi|>t\}\asymp
 \frac1{t^4(\log t)^2}.
\]
It has finite fourth moment.  Among \(n^2\) independent entries in a fresh
row, the probability of an entry of order \(\sqrt n\) is comparable to
\((\log n)^{-2}\).  These probabilities tend to zero but are not summable;
if the rows are independent, the second Borel--Cantelli lemma produces such
spikes infinitely often.  Since the operator norm dominates the magnitude
of every normalized entry, almost-sure edge convergence can fail while the
fixed-test in-probability conclusion remains valid.  This is exactly why the
nested theorem and the nonnested theorem have different modes of
convergence.

\section{Finite-separable continuous profiles}
\label{sec:profiles}

We apply the finite-type theorem to a class of continuously weighted
non-Hermitian matrices.  The point of the application is not to approximate
an arbitrary covariance kernel.  Rather, it shows that deterministic
left--right weights and literal transpose weights may be incorporated at the
exact fourth-moment threshold without losing the strong limit.

Let
\[
 T_n=\operatorname{diag}(1/n,2/n,\ldots,n/n)
\]
and let \(t\in C([0,1])\) be the coordinate function.  For every scalar
polynomial \(p\),
\[
 \norm{p(T_n)}\longrightarrow\norm{p(t)}_\infty,
 \qquad
 \tr_n p(T_n)\longrightarrow\int_0^1p(s)\,ds.
\]
The same maximum-on-a-dense-grid argument works at every fixed matrix level.
Thus \(T_n\) converges strongly to \(t\) in the faithful tracial
\(C^*\)-probability space
\begin{equation}
 \left(C([0,1]),\ f\mapsto\int_0^1f(s)\,ds\right).
 \label{eq:coordinate-base-strong}
\end{equation}

We use the homogeneous case \(r=1\) of the model.  For a pair atom write
\[
 V=(z_+,z_-)\in\C^d\oplus\C^d,
 \qquad
 z_+=(\xi_{ij}^{(\kappa)})_{\kappa\leq d},\quad
 z_-=(\xi_{ji}^{(\kappa)})_{\kappa\leq d}.
\]
Let \(e=(e_\kappa)_{\kappa\leq d}\) be its covariance-matched free
Gaussian limit from \cref{sec:model-results}.  If
\[
 u_\kappa=\sum_\ell\alpha_{\kappa\ell}g_{+,\ell},
 \qquad
 v_\kappa=\sum_\ell\beta_{\kappa\ell}g_{-,\ell}
\]
are the symmetric and antisymmetric real Gaussian factorizations, then
\begin{align}
 e_\kappa
 &=\frac1{\sqrt2}\left(
   \sum_\ell\alpha_{\kappa\ell}s_{+,\ell}
   -i\sum_\ell\beta_{\kappa\ell}s_{-,\ell}\right),
 \label{eq:profile-direct-limit}\\
 e_\kappa^{\leftarrow}
 &=\frac1{\sqrt2}\left(
   \sum_\ell\alpha_{\kappa\ell}s_{+,\ell}
   +i\sum_\ell\beta_{\kappa\ell}s_{-,\ell}\right).
 \label{eq:profile-reverse-limit}
\end{align}
The arrow is notation for the endpoint-reversed coordinate.  It is not an
intrinsic transpose operation in the limiting algebra, and it is not an
independent copy of \(e_\kappa\).

Fix an output dimension \(s\).  For finitely many indices
\((\rho,\kappa,\ell)\), let
\[
 a_{\rho\kappa\ell},b_{\rho\kappa\ell},
 \widetilde a_{\rho\kappa\ell},
 \widetilde b_{\rho\kappa\ell}\in C([0,1]),
\]
where every number of summands is fixed independently of \(n\).  Define
\begin{align}
 Z_{\rho,n}
 &=\sum_{\kappa,\ell}
   a_{\rho\kappa\ell}(T_n)X_{\kappa,n}
   b_{\rho\kappa\ell}(T_n) \notag\\
 &\quad+\sum_{\kappa,\ell}
   \widetilde a_{\rho\kappa\ell}(T_n)X_{\kappa,n}^{\transpose}
   \widetilde b_{\rho\kappa\ell}(T_n),
 \label{eq:profile-matrix}\\
 z_\rho
 &=\sum_{\kappa,\ell}
   a_{\rho\kappa\ell}(t)e_\kappa
   b_{\rho\kappa\ell}(t) \notag\\
 &\quad+\sum_{\kappa,\ell}
   \widetilde a_{\rho\kappa\ell}(t)e_\kappa^{\leftarrow}
   \widetilde b_{\rho\kappa\ell}(t).
 \label{eq:profile-limit}
\end{align}

\begin{corollary}[Finite-separable profile convergence]
\label{cor:finite-separable-profile}
Assume that the homogeneous pair atoms satisfy the hypotheses of
\cref{thm:nested-main}.  Then, on one probability-one event, for every fixed
\(q\) and every fixed \(M_q(\C)\)-valued \(*\)-polynomial \(Q\),
\begin{equation}
 \norm{Q(T_n,Z_{1,n},\ldots,Z_{s,n})}
 \longrightarrow
 \norm{Q(t,z_1,\ldots,z_s)},
 \label{eq:profile-strong}
\end{equation}
and the corresponding normalized traces converge.  In the fixed-law
nonnested model of \cref{thm:triangular-main}, the same assertions hold in
probability, separately for every fixed \((q,Q)\).

If another deterministic tuple \(A_n\) is retained, the same conclusion
holds provided \((T_n,A_n)\) converges jointly strongly; separate strong
convergence of \(T_n\) and \(A_n\) is not sufficient.
\end{corollary}

\begin{proof}
Install the literal transposes before applying the main theorem.  Namely,
consider the \(2d\)-color tuple
\[
 (X_{1,n},\ldots,X_{d,n},
   X_{1,n}^{\transpose},\ldots,X_{d,n}^{\transpose}).
\]
The forward and reverse halves of its unordered-pair atom are
\[
 W_+=(z_+,z_-),\qquad W_-=(z_-,z_+).
\]
Equivalently, \((W_+,W_-)=L(z_+,z_-)\), where
\(L:\R^{4d}\to\R^{8d}\) is a fixed real-linear duplication map.  Thus
independence across unordered pairs and the fourth-moment hypothesis are
unchanged, while the enlarged covariance is the possibly singular matrix
\(L\Sigma_{\R}L^{\transpose}\).  Its Gaussian comparator is \(LG\), not an
independent enlargement.  Endpoint reversal fixes the symmetric sector and
changes the sign of the antisymmetric sector, which gives precisely
\eqref{eq:profile-direct-limit}--\eqref{eq:profile-reverse-limit} with the
same free sources.

The main theorem now gives joint strong convergence of
\((T_n,X_n,X_n^{\transpose})\) for polynomial functions of \(T_n\).  To
pass to continuous masks, choose scalar polynomials \(p,q\) with
\(\norm{a-p}_\infty,\norm{b-q}_\infty\leq\varepsilon\).  Functional
calculus gives
\begin{align*}
 &\norm{a(T_n)X_{\kappa,n}b(T_n)
       -p(T_n)X_{\kappa,n}q(T_n)}\\
 &\qquad\leq
 \varepsilon\norm{X_{\kappa,n}}\norm b_\infty
 +\norm p_\infty\norm{X_{\kappa,n}}\varepsilon.
\end{align*}
The identical estimate holds with the transpose because
\(\norm{X_{\kappa,n}^{\transpose}}=\norm{X_{\kappa,n}}\).  There are only
finitely many masks.  Coordinate norms are almost surely bounded in the
nested model and tight in the nonnested model, and the same estimates hold
in the limit algebra.  Approximate every mask, apply the main theorem to the
resulting polynomial, and then let \(\varepsilon\downarrow0\).  Finally use
\cref{lem:polynomial-continuity} for \(Q\); the trace assertion follows from
\(\abs{\tr(H)}\leq\norm H\).
\end{proof}

\begin{example}[A two-orientation profiled ensemble]
\label{ex:two-orientation-profile}
Take one color and let the homogeneous pair atom
\(V=(\xi_+,\xi_-)\) satisfy the hypotheses above.  The concrete matrix
\[
  Z_n=(I+T_n)X_nT_n+T_nX_n^{\transpose}(I-T_n)
\]
has entries
\[
 (Z_n)_{ij}=\frac1{\sqrt n}\left[
   (1+i/n)(j/n)\xi_{ij}+(i/n)(1-j/n)\xi_{ji}
 \right].
\]
Thus the direct and endpoint-reversed coordinates carry genuinely different
position-dependent weights.  The corollary identifies the strong limit as
\[
 z=(1+t)e\,t+t e^{\leftarrow}(1-t).
\]
In particular, every fixed matrix-valued polynomial in \((T_n,Z_n)\) has the
corresponding limiting norm.  Moreover, for each fixed \(\lambda\in\C\), the
singular-value set of \(Z_n-\lambda I\) converges in Hausdorff distance to the
spectrum of \(\abs{z-\lambda}\); equivalently, the self-adjoint hermitizations
of \(Z_n-\lambda I\) converge spectrally.  These convergences are almost sure
in the nested model and in probability for fixed-law fresh rows.  This last
statement follows from \cref{cor:spectral-convergence} and is not an assertion
about the eigenvalue support of the nonnormal matrix \(Z_n\).
\end{example}

Entrywise, \eqref{eq:profile-matrix} is exactly
\[
 (Z_{\rho,n})_{ij}=\frac1{\sqrt n}\sum_\kappa
 \left[
  k_{\rho\kappa}^{\mathrm{dir}}(i/n,j/n)\xi_{ij}^{(\kappa)}
 +k_{\rho\kappa}^{\mathrm{rev}}(i/n,j/n)\xi_{ji}^{(\kappa)}
 \right],
\]
where
\[
 k_{\rho\kappa}^{\mathrm{dir}}(x,y)
 =\sum_\ell a_{\rho\kappa\ell}(x)b_{\rho\kappa\ell}(y),
 \quad
 k_{\rho\kappa}^{\mathrm{rev}}(x,y)
 =\sum_\ell\widetilde a_{\rho\kappa\ell}(x)
                    \widetilde b_{\rho\kappa\ell}(y).
\]
This form also identifies the exact boundary of the corollary.  It permits
fixed finite sums of separable continuous masks generated from finitely many
homogeneous latent colors.  It does not assert a theorem for arbitrary
continuous kernels, separable rank growing with \(n\), or a continuum of
unrelated local pair laws.  Although finite sums of \(a(x)b(y)\) are dense in
\(C([0,1]^2)\), uniform kernel approximation alone does not give the
dimension-free Schur-multiplier bound needed to interchange that
approximation with \(n\to\infty\).  An extension beyond finite separable
rank therefore requires an additional stability estimate.

Finally, the enlarged unweighted Gaussian family
\((e,e^{\leftarrow})\) is free from the base algebra \(C([0,1])\).  The
weighted variables \(z_\rho\), which contain coefficients from that algebra,
should not themselves be described as scalar-free from it.

\section{Consequences and scope}
\label{sec:consequences}

We collect several immediate consequences and make explicit what the theorem
does not address.

\subsection{Spectra and singular values of fixed polynomials}

\begin{corollary}[Self-adjoint spectral convergence]
\label{cor:spectral-convergence}
Let \(H\) be a fixed self-adjoint matrix-valued \(*\)-polynomial in the
random matrices, type projections, and deterministic companions, and let
\(H_n\) and \(H_\infty\) denote its finite-dimensional and limiting
evaluations.  In the nested model,
\begin{equation}
 d_{\mathrm H}\bigl(\Spec(H_n),\Spec(H_\infty)\bigr)
 \longrightarrow0
 \quad\text{almost surely}.
 \label{eq:spectral-hausdorff}
\end{equation}
For fixed-law nonnested rows the convergence holds in probability.
Consequently, the singular-value sets of every fixed (not necessarily
self-adjoint) polynomial converge in the same mode.
\end{corollary}

\begin{proof}
Strong convergence gives norm convergence for every scalar polynomial in
\(H_n\).  Polynomial approximation of continuous functions then excludes
spectrum from compact sets disjoint from \(\Spec(H_\infty)\), proving the
upper inclusion.  For the reverse inclusion, take a nonnegative continuous
function supported in a small interval meeting \(\Spec(H_\infty)\) and
nonzero at a spectral point.  Faithfulness of the limiting trace implies
that its trace at \(H_\infty\) is positive.  Trace convergence, again after
polynomial approximation, forces \(H_n\) to have spectrum in that interval.
The same argument in probability proves the triangular version.  Apply the
self-adjoint assertion to \(Q_n^*Q_n\) and use continuity of the square root
to obtain the singular-value statement.
\end{proof}

Thus the theorem rules out asymptotic operator-norm outliers for every fixed
self-adjoint polynomial covered by the model.  For a nonnormal polynomial,
however, norm and singular-value control alone do not imply convergence of
its eigenvalue support or Brown measure; such conclusions require
uniform quantitative control of very small singular values together with
logarithmic integrability, neither of which follows from the strong-convergence
statement proved here.

\subsection{A covariance universality principle}

The limiting family depends on a type-pair law only through its full real
second-order covariance.  In particular, two finite collections of pair
laws satisfying the hypotheses and having identical real covariance matrices
for every type pair have the same joint strong limit with every admissible
deterministic companion tuple.  The statement includes correlations among
colors, complex pseudocovariances, and correlations between the two endpoint
orientations.  It is therefore stronger than entrywise variance matching,
but it still relies on independence across distinct unordered pairs.

The fourth-moment assumption is a threshold for this operator-norm statement,
not a source of quantitative concentration.  After truncation the bounded
comparison estimates are quantitative, but removal of the truncation under
only \(L^4\) yields qualitative convergence.  We make no exponential-tail or
finite-sample concentration claim for the unbounded model.

\subsection{Directions beyond the finite-type theorem}

The proof isolates three extensions that would require genuinely new input.
First, arbitrary continuous covariance kernels would require a
dimension-free stability estimate for approximating Schur multipliers; the
density of finite-separable functions is not by itself sufficient.  Second,
growing numbers of types, colors, separable terms, or matrix amplification
levels would require estimates uniform in those parameters.  Third, local
laws varying over a continuum of endpoint locations introduce both a
distributional-uniformity problem at the fourth-moment threshold and, for
nonexchangeable orientations, the discontinuous mask
\(\one_{\{x<y\}}\).  None of these uniform regimes is claimed here.

Within the finite regime, the reusable conclusion is instead exact: a
covariance-matched Gaussian endpoint, once identified jointly with the
deterministic base, transfers to all fixed amplified polynomial tests under a
finite-fourth-moment hypothesis, with the mode of convergence determined by
the across-size coupling.  That moment threshold is sharp for the nested
almost-sure statement already on the real Wigner subfamily.  At the same
threshold, arbitrary fresh triangular rows do not admit a coupling-invariant
almost-sure upgrade; our fixed-law nonnested conclusion remains convergence in
probability.

\appendix
\section{Covariance and normalization calculations}
\label{app:covariance-bookkeeping}

We collect here the finite-dimensional calculations behind
Section~\ref{sec:gaussian-comparator}.  Besides verifying constants, they
show explicitly that no complex covariance information and no type-proportion
factor is lost in the corner realization.

\subsection{Real covariance and directed cross-type blocks}

Let \(Z=(Z_+,Z_-)\in\C^{2d}\) be centered and let
\(\Sigma_{\mathbb R}=\E[\mathfrak r(Z)\mathfrak r(Z)^{\mathsf T}]\).
Choose \(C\in\R^{4d\times L}\), \(L\le4d\), with
\(CC^{\mathsf T}=\Sigma_{\mathbb R}\).  Combining each pair of real rows
of \(C\) into a complex row gives coefficients
\((\alpha_{\kappa\ell},\beta_{\kappa\ell})\) such that
\[
 (Z_{+,\kappa},Z_{-,\kappa})_{\kappa\le d}
 \stackrel{\mathrm{law}}=
 \sum_{\ell=1}^{L}
  (\alpha_{\kappa\ell},\beta_{\kappa\ell})_{\kappa\le d}g_\ell.
\]
This is a factorization on \(\R^{4d}\), not merely a factorization of the
Hermitian covariance \(\E[ZZ^*]\).  For example, it determines at once all
four families
\begin{align*}
 \E[Z_{+,\kappa}\overline{Z_{+,\lambda}}]
 &=\sum_\ell\alpha_{\kappa\ell}
                      \overline{\alpha_{\lambda\ell}},&
 \E[Z_{+,\kappa}Z_{+,\lambda}]
 &=\sum_\ell\alpha_{\kappa\ell}\alpha_{\lambda\ell},\\
 \E[Z_{+,\kappa}\overline{Z_{-,\lambda}}]
 &=\sum_\ell\alpha_{\kappa\ell}
                      \overline{\beta_{\lambda\ell}},&
 \E[Z_{+,\kappa}Z_{-,\lambda}]
 &=\sum_\ell\alpha_{\kappa\ell}\beta_{\lambda\ell},
\end{align*}
as well as the analogous formulas with the two endpoints interchanged.

Fix vertices \(i,j\) with \(\theta(i)=a\), \(\theta(j)=b\), and \(a<b\).
In \eqref{eq:cross-gaussian-matrix}, only the first summand contributes at
\((i,j)\) and only the second contributes at \((j,i)\).  Hence
\begin{equation}
 \left(\sqrt nY^{ab}_{\kappa,n}(i,j),
       \sqrt nY^{ab}_{\kappa,n}(j,i)\right)_{\kappa\le d}
 =\sum_{\ell=1}^{L_{ab}}
  (\alpha^{ab}_{\kappa\ell},\beta^{ab}_{\kappa\ell})_{\kappa\le d}
  \sqrt nG^{ab}_{\ell,n}(i,j).
 \label{eq:appendix-cross-entry-vector}
\end{equation}
The right-hand side has exactly the real covariance of
\((Z^{ab}_+,Z^{ab}_-)\).  If \(j<i\), reversal consistency of the laws swaps
the two displayed components.  Different unordered edges use distinct
Ginibre entries, which verifies independence as well as covariance.

If \(C_1C_1^{\mathsf T}=C_2C_2^{\mathsf T}=\Sigma_{\mathbb R}\), pad the
two matrices with zero columns to the same width.  There is an orthogonal
matrix \(O\) on the source space such that \(C_2=C_1O\) after restriction to
the support of \(\Sigma_{\mathbb R}\).  Both standard real Gaussian vectors
and free circular source families are invariant under this real orthogonal
rotation.  This gives a direct finite-dimensional proof of
\Cref{lem:gaussian-factorization-invariance}.

\subsection{Exchangeable endpoints and the factor
\texorpdfstring{\(2^{-1/2}\)}{2 to the power -1/2}}

Let \(J(z_+,z_-)=(z_-,z_+)\), and suppose the centered real Gaussian vector
\(Z=(Z_+,Z_-)\) satisfies \(JZ\stackrel{\mathrm{law}}=Z\).  Under the
orthogonal change of variables
\[
 U=\frac{Z_++Z_-}{\sqrt2},
 \qquad V=\frac{Z_+-Z_-}{\sqrt2},
\]
the action of \(J\) becomes \((U,V)\mapsto(U,-V)\).  Thus, for any real
linear functionals \(f,h\),
\[
 \E[f(U)h(V)]=\E[f(U)h(-V)]=-\E[f(U)h(V)]=0.
\]
The real covariance matrix of \((U,V)\) is block diagonal, and Gaussianity
therefore makes the two blocks independent.

For independent standard real Gaussians \(g_+,g_-\), the \(2\)-by-\(2\)
linear transformation back to the endpoints is
\begin{equation}
 \binom{Z_+}{Z_-}
 =\frac1{\sqrt2}
  \begin{pmatrix}1&1\\1&-1\end{pmatrix}
  \binom{U}{V}.
 \label{eq:appendix-hadamard-transform}
\end{equation}
At a same-type matrix edge, a symmetric source places the same copy of
\(U/\sqrt n\) in the two directions, whereas a skew source places opposite
copies of \(V/\sqrt n\).  Equation
\eqref{eq:appendix-hadamard-transform} consequently forces the outer factor
\(2^{-1/2}\) in \eqref{eq:same-gaussian-matrix}.  Removing it would double
every real covariance in that block.

On the diagonal, the symmetric matrix in
\eqref{eq:symmetric-skew-ginibre} has entries
\(S_{ii}=\sqrt2G_{ii}\), whereas \(K_{ii}=0\).  These entries need not have
the comparator's prescribed diagonal covariance.  Coupling the off-diagonal
parts and subtracting the two diagonal Gaussian vectors produces
\(D_{\kappa,n}\) in \eqref{eq:gaussian-diagonal-correction}.  If
\(\sigma^2\le C/n\) bounds every real-coordinate variance, then the scalar
Gaussian tail bound and a union bound give
\[
 \mathbb P\{\norm{D_{\kappa,n}}>t\}
 \le 4n\exp\!\left(-\frac{nt^2}{C'}\right),
\]
which is summable for each fixed \(t>0\).  This proves
\eqref{eq:diagonal-gaussian-negligible} without requiring the correction to
be independent of the polynomial model.

\subsection{The doubled corner and its trace}

Write a GOE matrix in \(n\)-by-\(n\) blocks as
\[
 W_{2n}=\begin{pmatrix}W_{11}&W_{12}\\W_{12}^{\mathsf T}&W_{22}\end{pmatrix}.
\]
Every entry of \(W_{12}\), including \((W_{12})_{ii}\), is an off-diagonal
entry of \(W_{2n}\).  The unordered global index pairs
\(\{i,n+j\}\), \(1\le i,j\le n\), are all distinct.  Thus
\(G_n=\sqrt2W_{12}\) has independent \(N(0,n^{-1})\) entries.  A direct
block multiplication gives
\[
 \sqrt2e^{(n)}_{11}W_{2n}e^{(n)}_{22}e^{(n)}_{21}
 =\begin{pmatrix}\sqrt2W_{12}&0\\0&0\end{pmatrix},
\]
which proves \eqref{eq:ginibre-goe-corner}.

Likewise, if \(Z\in\M_n\), then
\[
 \norm{\begin{pmatrix}Z&0\\0&0\end{pmatrix}}=\norm{Z},
 \qquad
 \frac1{2n}\Tr\begin{pmatrix}Z&0\\0&0\end{pmatrix}
 =\frac12\frac1n\Tr Z.
\]
This is \eqref{eq:corner-norm-trace}.  In the limit, the corner projection
\(e\) has trace \(1/2\), so the normalized corner trace is
\(\varphi_e=\varphi(e)^{-1}\varphi=2\varphi\).

To verify the normalization of \eqref{eq:circular-corner}, let
\(c=\sqrt2esfv\), where \(s\) is standard semicircular and free from
\(\M_2\otimes B\).  For \(b\in B\), identified with \(ebe\), the
semicircular covariance identity gives
\begin{align*}
 E_e(cbc^*)
 &=2eE\!\left(s(vbv^*)s\right)e
   =2e\,(\tr_2\otimes\tau_B)(fb)\,e
   =\tau_B(b)e,\\
 E_e(c^*bc)
 &=2v^*E\!\left(s(ebe)s\right)v
   =\tau_B(b)e,
\end{align*}
whereas \(E_e(cbc)=E_e(c^*bc^*)=0\) because the intervening matrix-unit
coefficient has zero trace.  Higher cumulants vanish.  Hence \(c\) is
standard circular in the normalized corner, not a circular element of
variance \(1/2\).

\subsection{Type proportions are already encoded by compression}

There is no additional factor \(\sqrt{\pi_a}\) or
\(\sqrt{\pi_b}\) in \eqref{eq:cross-free-limit}.  To see this directly,
let \(E_B\) be the trace-preserving conditional expectation onto \(B\) in
the free-product realization.  For the cross-type component and \(b\in B\),
the circular covariance relations give
\begin{align}
 E_B(x^{ab}_\kappa b(x^{ab}_\lambda)^*)
 &=\sum_\ell
   \alpha^{ab}_{\kappa\ell}
   \overline{\alpha^{ab}_{\lambda\ell}}
   p_a\tau_B(p_bbp_b)p_a
 \notag\\[-2mm]
 &\quad+\sum_\ell
   \beta^{ab}_{\kappa\ell}
   \overline{\beta^{ab}_{\lambda\ell}}
   p_b\tau_B(p_abp_a)p_b,
 \label{eq:cross-B-valued-star-covariance}\\
 E_B(x^{ab}_\kappa bx^{ab}_\lambda)
 &=\sum_\ell
   \alpha^{ab}_{\kappa\ell}\beta^{ab}_{\lambda\ell}
   p_a\tau_B(p_bbp_b)p_a
 \notag\\[-2mm]
 &\quad+\sum_\ell
   \beta^{ab}_{\kappa\ell}\alpha^{ab}_{\lambda\ell}
   p_b\tau_B(p_abp_a)p_b.
 \label{eq:cross-B-valued-pseudo-covariance}
\end{align}
At \(b=\one\), the factors
\(\tau_B(p_a)=\pi_a\) and \(\tau_B(p_b)=\pi_b\) appear automatically.
They are the limiting fractions of admissible column indices in the
corresponding finite blocks.

For a same-type component, the analogous formulas are
\begin{align}
 E_B(x^{aa}_\kappa b(x^{aa}_\lambda)^*)
 &=\frac12\left(
   \sum_\ell\alpha^a_{\kappa\ell}
                   \overline{\alpha^a_{\lambda\ell}}
  +\sum_\ell\beta^a_{\kappa\ell}
                   \overline{\beta^a_{\lambda\ell}}
  \right)p_a\tau_B(p_abp_a)p_a,
 \label{eq:same-B-valued-star-covariance}\\
 E_B(x^{aa}_\kappa bx^{aa}_\lambda)
 &=\frac12\left(
   \sum_\ell\alpha^a_{\kappa\ell}\alpha^a_{\lambda\ell}
  -\sum_\ell\beta^a_{\kappa\ell}\beta^a_{\lambda\ell}
  \right)p_a\tau_B(p_abp_a)p_a.
 \label{eq:same-B-valued-pseudo-covariance}
\end{align}
These are exactly the star-covariance and pseudocovariance obtained from
the Hadamard transform \eqref{eq:appendix-hadamard-transform}.  Mixed
covariances between distinct type-pair components vanish by freeness.  Thus
the displayed limit retains the full real covariance and has precisely the
normalization dictated by the \(n^{-1/2}\) entry scale.

\section{Amplification and target-independent linearization}
\label{app:amplification-linearization}

We give the operator-algebraic details behind
Lemma~\ref{lem:fixed-amplification} and
Proposition~\ref{prop:affine-pencil-criterion}.  Throughout this appendix,
all matrix levels and all polynomials are fixed before the matrix dimension
tends to infinity.

\subsection{The ultraproduct argument}

\begin{proof}[Proof of Lemma~\ref{lem:fixed-amplification}]
Strong convergence implies \(\sup_n\norm{Y_{j,n}}<\infty\) for every
coordinate.  Fix a free ultrafilter \(\omega\) on \(\N\) and form the
\(C^*\)-ultraproduct
\[
 \mathfrak M_\omega
 =\prod_{n}\M_{N_n}(\C)\big/
   \bigl\{(T_n):\lim_{n\to\omega}\norm{T_n}=0\bigr\}.
\]
For a scalar \(*\)-polynomial \(P\), define on the polynomial algebra
generated by \(y\)
\begin{equation}
 \pi_\omega(P(y))=[P(Y_n)]_\omega.
 \label{eq:ultraproduct-homomorphism}
\end{equation}
If \(P(y)=0\), then
\(\norm{P(Y_n)}\to0\), so the definition is independent of the polynomial
representative.  Moreover,
\[
 \norm{\pi_\omega(P(y))}
 =\lim_{n\to\omega}\norm{P(Y_n)}
 =\norm{P(y)}.
\]
It follows that \(\pi_\omega\) extends to an injective unital
\(*\)-homomorphism
\[
 C^*(y_1,\ldots,y_s,\one)\longrightarrow\mathfrak M_\omega.
\]
Every injective \(*\)-homomorphism of \(C^*\)-algebras is completely
isometric.  Using the canonical identification of the fixed matrix level
\(\M_q(\mathfrak M_\omega)\) with the corresponding matrix ultraproduct,
we obtain, for each fixed \(q\) and
\(Q\in\M_q(\C\langle z,z^*\rangle)\),
\begin{equation}
 \lim_{n\to\omega}\norm{Q(Y_n)}
 =\norm{(\operatorname{id}_{\M_q}\otimes\pi_\omega)(Q(y))}
 =\norm{Q(y)}.
 \label{eq:ultralimit-amplified-norm}
\end{equation}
The ultrafilter \(\omega\) was arbitrary.  A bounded scalar sequence whose
limit along every free ultrafilter is the same number converges ordinarily
to that number.  This proves
\eqref{eq:amplified-norm-convergence}.  Finally, if
\(Q=(Q_{ab})_{a,b\leq q}\), then
\[
 (\tr_q\otimes\tr_{N_n})Q(Y_n)
 =\frac1q\sum_{a=1}^q\tr_{N_n}Q_{aa}(Y_n),
\]
and scalar trace convergence gives
\eqref{eq:amplified-trace-convergence}.  The argument fixes \(q\) before
taking \(n\to\infty\), and yields no uniform control as \(q\) varies.
\end{proof}

\subsection{Quotient spectra and the polynomial upper bound}

Let
\[
 \mathfrak M_\infty
 =\prod_n\M_{N_n}(\C)\big/\bigoplus_n\M_{N_n}(\C),
\]
where the ideal consists of norm-null sequences, and let
\(\widehat Y_j=[(Y_{j,n})_n]\).  The coordinate pencils in
Proposition~\ref{prop:affine-pencil-criterion} show first that each coordinate
sequence is bounded: take \(q=1\), \(A_0=0\), and a single nonzero
coefficient.  For a bounded self-adjoint sequence \((T_n)\),
\begin{equation}
 \Spec([(T_n)])
 =\bigcap_{m\geq1}\overline{\bigcup_{n\geq m}\Spec(T_n)}.
 \label{eq:spectrum-sequence-quotient}
\end{equation}
Indeed, outside the set on the right the resolvents
\((T_n-\lambda)^{-1}\) are eventually uniformly bounded; changing finitely
many coordinates then gives an inverse in the quotient.  Conversely, a
sequence of spectral points tending to \(\lambda\) prevents a uniformly
bounded inverse modulo the norm-null ideal.

It follows from \eqref{eq:pencil-spectral-inclusion} and
\eqref{eq:spectrum-sequence-quotient} that, for every fixed self-adjoint
matrix coefficient list,
\begin{equation}
 \Spec\!\left(A_0\otimes\one+
      \sum_{j=1}^sA_j\otimes\widehat Y_j\right)
 \subseteq
 \Spec\!\left(A_0\otimes\one+
      \sum_{j=1}^sA_j\otimes y_j\right).
 \label{eq:exact-quotient-pencil-inclusion}
\end{equation}

For completeness, we explain the target-independent role of
linearization.  Given a matrix-valued \(*\)-polynomial \(P\), first replace
it by the self-adjoint polynomial \(\mathfrak h(P)\) from
\eqref{eq:hermitization-polynomial}.  The self-adjoint linearization trick
\cite[Lemma~1 and pp.~758--760]{HaagerupThorbjornsen2005} associates to
\(\mathfrak h(P)-\lambda\) a larger self-adjoint affine pencil whose
coefficients depend only on \(P\) and \(\lambda\).  Schur complementation
relates invertibility of this pencil to invertibility of
\(\mathfrak h(P)-\lambda\).  Therefore
\eqref{eq:exact-quotient-pencil-inclusion} transfers every resolvent point
of \(\mathfrak h(P(y))\) to the corresponding evaluation at
\(\widehat Y\).  Equivalently,
\begin{equation}
 \norm{P(\widehat Y)}\leq\norm{P(y)}.
 \label{eq:quotient-polynomial-upper-bound}
\end{equation}
Since the norm of a class in \(\mathfrak M_\infty\) is the limsup of the
representative norms, this is precisely
\eqref{eq:pencil-upper-bound}.  Notice that the construction uses no
structural property of \(y\) beyond being a tuple of operators.  In
particular, it is not tied to a semicircular or Gaussian target.

\subsection{Polarization and the faithful-trace lower bound}

We next justify the trace passage in
\eqref{eq:polynomial-trace-from-pencils}.  First fix a self-adjoint linear
pencil without constant term,
\[
 K_n=\sum_{j=1}^s A_j\otimes Y_{j,n},
 \qquad K=\sum_{j=1}^s A_j\otimes y_j.
\]
For any integer \(k\geq0\), choose \(p\) with \(2p\geq k\) and apply
\eqref{eq:pencil-even-moment-convergence} to \(tI+K_n\).  The coefficient
of \(t^{2p-k}\) in its normalized \(2p\)-th moment is
\[
 \binom{2p}{k}(\tr_q\otimes\tr_{N_n})(K_n^k).
\]
Since these are polynomials of uniformly bounded degree in \(t\), their
pointwise convergence for \(2p+1\) distinct real values of \(t\) implies
coefficientwise convergence.  We have therefore obtained convergence of
all powers of every self-adjoint linear pencil.

It remains to recover ordered words rather than only symmetrized products.
Let
\(w=z_{i_1}\cdots z_{i_k}\) with \(k\geq2\), and regard the indices
of \(\M_k(\C)\) cyclically.  For
\(\zeta=(\zeta_1,\ldots,\zeta_k)\in\mathbb T^k\), define self-adjoint
coefficient matrices
\begin{equation}
 A_j(\zeta)=
 \sum_{\ell:i_\ell=j}
 \left(\zeta_\ell E_{\ell,\ell+1}
       +\overline{\zeta_\ell}E_{\ell+1,\ell}\right),
 \label{eq:matrix-unit-polarization}
\end{equation}
where \(E_{k,k+1}=E_{k,1}\).  In the normalized trace of
\(\bigl(\sum_jA_j(\zeta)\otimes Y_{j,n}\bigr)^k\), the Fourier coefficient of
\(\zeta_1\cdots\zeta_k\) is exactly
\begin{equation}
 \tr_{N_n}\bigl(Y_{i_1,n}\cdots Y_{i_k,n}\bigr).
 \label{eq:word-fourier-coefficient}
\end{equation}
Indeed, the only nonzero products using every forward matrix unit once are
the \(k\) cyclic traversals of the directed cycle; their matrix traces
contribute a factor \(k\), canceled by \(\tr_k\), and cyclicity of
\(\tr_{N_n}\) identifies their word traces.  The coordinate bounds obtained
above give a uniform bound on these trigonometric polynomials, so dominated
convergence in the Fourier integral yields convergence of the trace of
\(w(Y_n)\).
Words of length zero or one are immediate.  By linearity, scalar polynomial
trace convergence follows, and matrix-valued trace convergence is entrywise
as in the proof of Lemma~\ref{lem:fixed-amplification}.  This proves
\eqref{eq:polynomial-trace-from-pencils}.

Finally, fix \(P\in\M_q(\C\langle z,z^*\rangle)\) and set
\[
 a=P(y)^*P(y)\in\M_q(\mathcal A)_+.
\]
For every \(p\geq1\), positivity and normalized trace give
\begin{align*}
 \norm{P(Y_n)}
 &\geq
 \left[(\tr_q\otimes\tr_{N_n})
       \bigl(P(Y_n)^*P(Y_n)\bigr)^p\right]^{1/(2p)}.
\end{align*}
Taking the liminf and using the polynomial trace convergence proves
\eqref{eq:faithful-trace-lower-bound}.  The trace
\(\tr_q\otimes\tau\) is faithful on \(\M_q(\mathcal A)\).  Hence, for
every positive \(a\),
\begin{equation}
 \lim_{p\to\infty}
 \bigl[(\tr_q\otimes\tau)(a^p)\bigr]^{1/p}=\norm{a}.
 \label{eq:faithful-lp-to-operator-norm}
\end{equation}
To see this directly, choose a continuous function supported in
\((\norm{a}-\varepsilon,\norm{a}]\) and nonzero at \(\norm{a}\).  Its
functional-calculus value at \(a\) is nonzero, so faithfulness gives positive
spectral mass in that
interval.  Letting first \(p\to\infty\) and then \(\varepsilon\downarrow0\)
proves \eqref{eq:faithful-lp-to-operator-norm}.  Thus
\eqref{eq:faithful-trace-lower-bound} yields
\[
 \liminf_{n\to\infty}\norm{P(Y_n)}\geq\norm{P(y)}.
\]
Combining this inequality with
\eqref{eq:quotient-polynomial-upper-bound} proves
Proposition~\ref{prop:affine-pencil-criterion}.

\section{Precise forms of the external probabilistic inputs}
\label{app:source-interfaces}

For reference, we record the exact forms in which the four external
probabilistic results enter the proof.  This also separates their roles: the
matrix comparison theorem treats bounded independent summands, the GOE result
identifies the Gaussian strong limit, and the two edge theorems control the
fourth-moment residuals.

\subsection{Brailovskaya--van Handel comparison}

The theorem and lemma numbering in this subsection follows the authors'
arXiv version \texttt{2201.05142}, the source text used for the interface
check; the bibliography records the corresponding published article.

The model in \cite[Section~2.1]{BvH2024} is
\[
 X=Z_0+\sum_{j=1}^N Z_j,
\]
where \(Z_0\) is an arbitrary deterministic self-adjoint matrix and the
\(Z_j\)'s are independent centered self-adjoint random matrices.  They need
not have identical distributions.  The Gaussian comparator \(G\) has the
same mean and the full covariance
\[
 \Cov(X)_{ij,k\ell}
 =\E\bigl[(X-\E X)_{ij}
          \overline{(X-\E X)_{k\ell}}\bigr].
\]
For a self-adjoint matrix this complex covariance determines the covariance
of all real and imaginary entry coordinates, so correlations internal to one
unordered-pair summand are retained.

The three parameters used in \cref{sec:bounded-transfer} are
\begin{align*}
 \sigma(X)^2&=\norm{\E(X-\E X)^2},\\
 \sigma_*(X)^2&=
 \sup_{\norm u=\norm v=1}\E\abs{\ip{u}{(X-\E X)v}}^2,\\
 R(X)&=\left\|\max_j\norm{Z_j}\right\|_{L^\infty}.
\end{align*}
They depend only on the centered random part and not on \(Z_0\).  Thus a
dense deterministic companion pencil is permitted, provided that all random
terms attached to one unordered pair are first grouped as in
\eqref{eq:typed-pair-summand}.

At ambient dimension \(D\), \cite[Theorem~2.6]{BvH2024} gives, for every
\(t\geq0\),
\[
 \mathbb P\left\{d_{\rm H}(\Spec X,\Spec G)>
 C\left(\sigma_*\sqrt t+R^{1/3}\sigma^{2/3}t^{2/3}+Rt\right)\right\}
 \leq De^{-t}.
\]
The assertion holds for every coupling of \(X\) and \(G\).  For
\(p\in\N\) and \(2p\leq q_{\rm mom}\leq\infty\), the first bound in
\cite[Theorem~2.9]{BvH2024} is
\begin{multline*}
 \left|\bigl(\E\tr_DX^{2p}\bigr)^{1/(2p)}
 -\bigl(\E\tr_DG^{2p}\bigr)^{1/(2p)}\right|\\
 \lesssim
 R_{q_{\rm mom}}(X)^{1/3}
 \sigma_{q_{\rm mom}}(X)^{2/3}p^{2/3}
 +R_{q_{\rm mom}}(X)p,
\end{multline*}
with \(R_\infty=R\) and \(\sigma_\infty=\sigma\).  Finally,
\cite[Lemma~9.20]{BvH2024} states that, for \(t\geq p\),
\begin{multline*}
 \mathbb P\left\{
 \left|\bigl(\tr_DX^{2p}\bigr)^{1/(2p)}
 -\bigl(\E\tr_DX^{2p}\bigr)^{1/(2p)}\right|
 \right.\\
 \left.\hspace{8mm}\geq
 C\left(\sigma_*+R^{1/2}
       \bigl(\E\tr_DX^{2p}\bigr)^{1/(4p)}\right)\sqrt t+CRt
 \right\}\leq2e^{-t}.
\end{multline*}
The normalization is \(\tr_D=D^{-1}\Tr\).  In our application \(D=qn\)
with \(q\) fixed, which gives precisely
\eqref{eq:bvh-spectrum-rate}--\eqref{eq:bvh-moment-concentration-rate}.

The linearization proposition displayed in \cite[Proposition~9.18]{BvH2024} is
formulated for a semicircular target without an adjoining deterministic
tuple.  We therefore use it neither to introduce companions nor to identify
mixed companion norms.  Those steps are supplied by the target-independent
\cref{prop:affine-pencil-criterion}, whose affine tuple contains both the
random coordinates and all deterministic companions.

\subsection{Fan--Sun--Wang and the doubled GOE model}

The theorem numbering in this subsection follows the authors' arXiv version
\texttt{1903.09592}, the source text used for the interface check; the
bibliography records the corresponding published article.

In \cite[Theorem~4.3]{FanSunWang2021}, the random matrices are independent
GOE matrices of size \(N\), with off-diagonal variance \(1/N\) and diagonal
variance \(2/N\).  If a deterministic tuple converges jointly strongly, then
adjoining these GOE matrices gives almost-sure joint trace and norm
convergence to a free semicircular family adjoined to the deterministic
limit.  The source explicitly permits the deterministic matrices to lie in
\(\M_N(\C)\); it does not require them to be real symmetric.  The norm
assertion is stated for self-adjoint polynomials.  For an arbitrary polynomial
\(Q\), the same conclusion follows
by applying it to the self-adjoint polynomial \(Q^*Q\); trace convergence
follows by separating real and imaginary parts.

We apply this result at \(N=2n\) to one deterministic tuple containing the
four \(2\times2\) matrix units, \(I_2\otimes P_{a,n}\), and
\(I_2\otimes A_{j,n}\).  If \(W_{2n}\) has the above GOE normalization, then
\[
 \sqrt2\,(E_{11}\otimes I_n)W_{2n}(E_{22}\otimes I_n)
\]
has, in its \((1,2)\) corner, an entire real Ginibre matrix with independent
\(N(0,1/n)\) entries.  This includes the entries with matching within-corner
indices: they remain globally off-diagonal entries of the \(2n\)-dimensional
GOE.  Compression is contractive, and identifying a compressed corner with
\(M_n(\C)\) is isometric; on the first corner
\[
 \tr_n=2\tr_{2n}.
\]
These are the normalization and trace factors used in
\cref{sec:gaussian-comparator}.

The theorem of Fan--Sun--Wang is invoked only at scalar matrix level.  Passage
to every separately fixed coefficient size is the complete-isometry argument
of \cref{lem:fixed-amplification}; no uniformity for a growing amplification
is asserted.

\subsection{Bai--Yin for cross-type residuals}

The \(k=1\) case of \cite[Theorem~2.1]{BaiYin1986} implies that, for one
nested infinite real iid array with centered entry law of variance
\(\sigma^2\) and finite fourth moment,
\[
 \limsup_{n\to\infty}\norm{n^{-1/2}W_n}\leq2\sigma
 \qquad\text{almost surely}.
\]
Applying the real result separately to the real and imaginary parts of a
centered complex entry \(z\) gives
\[
 \limsup_n\norm{n^{-1/2}Z_n}
 \leq2\bigl(\norm{\Rea z}_2+\norm{\Ima z}_2\bigr)
 \leq2\sqrt2\norm z_2.
\]
Independence of the real and imaginary parts is not used.  In a typed
cross-block, this full iid array is only an auxiliary completion; orthogonal
compression recovers the observed block.  The two directed orientations are
completed in separate arrays, so their prescribed correlation inside the
original pair atom is not replaced by independence.

\subsection{Anderson for same-type symmetric and skew residuals}

For one nested real symmetric Wigner array, \cite[Theorem~3]{Anderson2013}
restates the Bai--Yin conclusion
\[
 \lambda_{\max}(W_n/\sqrt n)\longrightarrow2\sigma
 \qquad\text{almost surely}.
\]
Applying the theorem to \(W_n\) and \(-W_n\), and intersecting the two
probability-one events, gives
\(\norm{W_n/\sqrt n}\to2\sigma\).  The centered diagonal hypothesis in that
statement is satisfied in our auxiliary completions because their diagonal
is identically zero.

The skew coordinates use the explicit twisted model in
\cite[Remark~4]{Anderson2013}: from a real symmetric upper-triangular array it
forms the Hermitian matrix \(iA_n\), where \(A_n\) is real skew-symmetric.
Apply the same remark to the two auxiliary upper-triangular laws \(q\) and
\(-q\).  The outputs are \(iA_n\) and \(-iA_n\), so the two-sign argument
gives
\[
 \norm{A_n/\sqrt n}=\norm{iA_n/\sqrt n}\longrightarrow2\sigma
 \qquad\text{almost surely}.
\]
If a coordinate has variance zero, it vanishes almost surely after centering;
this is also the degenerate case singled out in
\cite[Remark~3]{Anderson2013}.  The argument is applied one real coordinate
at a time and therefore does not invoke the independence assumptions from
Anderson's broader complex-Wigner setup.  This is why the four-coordinate
decomposition in \cref{sec:exact-l4} permits arbitrary dependence among the
real and imaginary parts of the original pair atom.

\section*{Declarations}

\paragraph{Data availability.}
No datasets were generated or analyzed in this theoretical work.

\paragraph{Conflict of interest.}
The authors declare no conflict of interest.

\paragraph{AI disclosure.}
This paper was prepared with the assistance of AI-powered academic writing
tools.  The workflow included literature-search strategy design, structure
planning, draft writing, citation verification, and formatting.  The authors
remain solely responsible for all mathematical content, arguments, citations,
and conclusions, and for the accuracy and integrity of the work.

{\footnotesize
\bibliographystyle{alpha}
\bibliography{references}
}

\end{document}